\documentclass[11pt]{amsart}
\usepackage{amsmath}
\usepackage{amssymb}
\usepackage{amsthm}
\usepackage{latexsym}
\usepackage{graphicx}
\usepackage{hyperref}
\usepackage{enumerate}
\usepackage[all]{xy}

\newtheorem{thm}{Theorem}[section]
\newtheorem{cor}[thm]{Corollary}
\newtheorem{lem}[thm]{Lemma}
\newtheorem{prop}[thm]{Proposition}
\newtheorem{thmintro}{Theorem}

\theoremstyle{definition}
\newtheorem{defn}[thm]{Definition}

\newtheorem{ex}[thm]{Example}

\newcommand{\N}{\mathbb N}
\newcommand{\Z}{\mathbb Z}
\newcommand{\Q}{\mathbb Q}
\newcommand{\R}{\mathbb R}
\newcommand{\C}{\mathbb C}
\newcommand{\mf}{\mathfrak}
\newcommand{\mc}{\mathcal}
\newcommand{\mb}{\mathbf}
\newcommand{\mh}{\mathbb}
\def\Irr{{\rm Irr}}
\newcommand{\mr}{\mathrm}
\newcommand{\enuma}[1]{\begin{enumerate}[\textup{(}a\textup{)}] {#1} \end{enumerate}}

\newcommand{\nr}{\mathrm{nr}}
\newcommand{\unr}{\mathrm{unr}}

\newcommand{\Rep}{\mathrm{Rep}}
\newcommand{\End}{\mathrm{End}}

\begin{document}

\title[Hochschild homology of twisted crossed product algebras]{
Hochschild homology of twisted crossed \\
products and twisted graded Hecke algebras}
\date{\today}
\maketitle
\begin{center}
{\Large Maarten Solleveld} \\[1mm]
IMAPP, Radboud Universiteit Nijmegen\\
Heyendaalseweg 135, 6525AJ Nijmegen, the Netherlands \\
email: m.solleveld@science.ru.nl 
\end{center}
\vspace{4mm}

\begin{abstract}
Let $A$ be a $\C$-algebra with an action of a finite group $G$, let $\natural$ be a 2-cocycle 
on $G$ and consider the twisted crossed product $A \rtimes \C [G,\natural]$. We determine the 
Hochschild homology of $A \rtimes \C [G,\natural]$ for two classes of algebras $A$:
\begin{itemize}
\item rings of regular functions on nonsingular affine varieties,
\item graded Hecke algebras.
\end{itemize}
The results are achieved via algebraic families of (virtual) representations and include
a description of the Hochschild homology as module over the centre of 
$A \rtimes \C [G,\natural]$. This paper prepares for a computation of the Hochschild 
homology of the Hecke algebra of a reductive $p$-adic group.
\end{abstract}
\vspace{2mm}

\tableofcontents

\section*{Introduction}

Consider a finite group $G$ and a 2-cocycle $\natural : G \times G \to \C^\times$.
The twisted group algebra $\C [G,\natural]$ is the vector space with basis
$\{T_g : g \in G \}$ and multiplication
\[
T_g \cdot T_{g'} = \natural (g,g') T_{g g'} \qquad g,g' \in G.
\]
Such algebras arise for instance from a projective representation $\pi : G \to PGL_n (\C)$.
Even if $\pi$ cannot be linearized, one can always regard $\pi$ as a representation of 
a suitable twisted group algebra of $G$. The general aim of this paper is to make certain 
results for algebras involving $\C [G]$ available for similar algebras that involve 
$\C [G,\natural]$. In other words, we want to treat $\C[G,\natural]$ on the same footing as
the group algebra $\C [G]$. Although $\C [G,\natural]$ is always semisimple, this is not so
trivial, already because the dimensions of irreducible $\C [G,\natural]$-representations
depend on the image of $\natural$ in $H^2 (G,\C^\times)$. 

Let $A$ be a unital $\C$-algebra on which $G$ acts by algebra automorphisms. The twisted 
crossed product algebra $A \rtimes \C [G,\natural]$ is the vector space 
$A \otimes_\C \C [G,\natural]$ with multiplication rules
\begin{itemize}
\item $A$ and $\C [G,\natural]$ are embedded as subalgebras,
\item $T_g a T_g^{-1} = g (a)$ for $g \in G$ and $a \in A$.\\
\end{itemize}

\textbf{Twisted crossed products with rings of regular functions.}\\
Interesting examples of the above algebras arise when $V$ is a complex affine variety endowed
with a $G$-action and $A = \mc O (V)$, the ring of regular functions on $V$. Our motivation
to study algebras like $\mc O (V) \rtimes \C [G,\natural]$ stems from reductive $p$-adic groups. 
There twisted versions of Hecke algebras appear in several ways, see e.g. \cite{AMS3,Mor,SolEnd}. 
If one manually sets the $q$-parameters of such Hecke algebras equal to 1, one obtains an algebra 
of the form $\mc O (V) \rtimes \C [G,\natural]$. 

It is well-known that the irreducible representations of $\mc O (V) \rtimes G = 
\mc O (V) \rtimes \C [G]$ are naturally parametrized by
\begin{equation}\label{eq:2}
(V /\!/ G)_2 = \{ (v,\pi_v) : v \in V, \pi_v \in \Irr (G_v) \} / G ,
\end{equation}
where $g \cdot (v,\pi_v) = (gv, \pi_v \circ \mr{Ad}(g)^{-1})$. The parametrization map is simple:
\[
(v,\pi_v) \mapsto \mr{ind}_{\mc O (V) \rtimes G_v}^{\mc O (V) \rtimes G} (\C_v \otimes \pi_v) .
\]
Many invariants of $\mc O (V) \rtimes G$ are related to the space
\begin{equation}\label{eq:1}
(V /\!/ G)_1 = \{ (v,g) : v \in V, g \in G_v \} / G ,
\end{equation}
where $g \cdot (v,g') = (gv, g g' g^{-1})$. Indeed, for nonsingular $V$ the Hochschild homology
was computed by Brylinski and Nistor \cite{Bry,Nis}:
\begin{equation}\label{eq:3}
HH_n (\mc O (V) \rtimes G) = \big( \bigoplus\nolimits_{g \in G} \Omega^n (V) \big)^G =
\Omega^n \big( \{ (v,g) : v \in V, g \in G_v \} \big)^G . 
\end{equation}
Now we discuss our analogues with twisting by $\natural$. The same arguments as for \eqref{eq:2}
show that $\Irr (\mc O (V) \rtimes \C [G,\natural])$ is naturally parametrized by
\begin{equation}\label{eq:4}
(V /\!/ G)_\natural = \{ (v,\pi_v) : v \in V, \pi_v \in \Irr (\C [G_v,\natural]) \} / G ,
\end{equation}
where $g \cdot (v,\pi_v) = (gv, \pi_v \circ \mr{Ad}(T_g)^{-1}$. However, there is no direct
generalization of \eqref{eq:1}. To get around that, we define (for $g \in G$)
\[
\begin{array}{cccc}
\natural^g : & G & \to & \C^\times \\
& h & \mapsto & T_h T_g T_h^{-1} T_{hgh^{-1}}^{-1} 
\end{array} .
\]
Then $\natural^g \big|_{Z_G (g)}$ is character, and the $\natural^g$ measure how far away from
a group algebra $\C [G,\natural]$ is. Indeed, we check in Lemma \ref{lem:2.3} that
\[
\Irr (\C [G,\natural]) \quad \text{and} \quad 
\{ g \in G : \natural^g \big|_{Z_G (g)} = 1 \} / G\text{-conjugation}   
\]
have the same cardinality. This generalizes the well-known equality between the number of
conjugacy classes and the number of inequivalent irreducible representations of $G$.

Notice that $\mc O (V)^G$ is contained in the centre of $\mc O (V) \rtimes \C [G,\natural]$,
so that it acts naturally on the Hochschild homology of that algebra.

\begin{thmintro}\label{thm:A}
\textup{(see Theorem \ref{thm:2.1} and \eqref{eq:2.33})} \\
Let $V$ be a nonsingular complex affine variety with a $G$-action. There exists an isomorphism
of $\mc O (V)^G$-modules
\[
HH_n \big( \mc O (V) \rtimes \C [G,\natural] \big) \cong
\big( \bigoplus\nolimits_{g \in G} \Omega^n (V^g) \otimes \natural^g \big)^G . 
\]
\end{thmintro}

We can interpret \eqref{eq:3} as $HH_0 (\mc O (V) \rtimes G) = \mc O ( (V // G)_1 )$.
In contrast, it is not clear whether $HH_0 (\mc O (V) \rtimes \C [G,\natural] )$ is naturally
isomorphic to the coordinate ring of a complex affine variety. It is preferable to describe
this in noncommutative geometric terms. Then $\mc O (V) \rtimes \C [G,\natural]$ is the ring
``functions on the space $(V /\!/ G)_\natural$" and Theorem \ref{thm:A} describes the
``differential forms on $(V /\!/ G)_\natural$".\\[2mm]

\textbf{Twisted crossed products with graded Hecke algebras.}\\
Another class of algebras that we want to investigate is intrinsically non-commutative.
Let $\mh H (\mf t,W,k)$ be a graded Hecke algebra, where $W$ is a Weyl group acting
on a complex vector space $\mf t$ and $k$ is a real-valued parameter function. Let $\Gamma$
be a finite group acting on $\mh H (\mf t,W,k)$, such that all structure used to define 
$\mh H (\mf t,W,k)$ is preserved by the action. Given a 2-cocycle $\natural$ of $\Gamma$,
we build the twisted graded Hecke algebra
\begin{equation}\label{eq:5}
\mh H = \mh H (\mf t,W,k) \rtimes \C[\Gamma,\natural].
\end{equation}
For $k=0$, this algebra is just $\mc O (\mf t) \rtimes \C [W\rtimes \Gamma,\natural]$,
where we inflate $\natural$ to a 2-cocycle of $W \rtimes \Gamma$.
Algebras of the form \eqref{eq:5}, sometimes with a nontrivial $\natural$, play an important 
role in the study of parabolically induced representations of reductive $p$-adic groups
\cite{SolEnd}. That motivated us to determine their Hochschild homology.

It follows quickly from \cite{SolHomGHA} and Theorem \ref{thm:A} that as vector spaces
\begin{equation}\label{eq:6}
HH_n ( \mh H ) \cong \big( \bigoplus\nolimits_{w \in W\Gamma} 
\Omega^n (\mf t^w) \otimes \natural^w \big)^{W\Gamma} ,
\end{equation}
see \eqref{eq:4.6}. The nontrivial content of this statement is that for every element on 
the right hand side, a particular representative in a differential complex computing 
$HH_n (\mh H)$ is exhibited. However, usually \eqref{eq:6} is not an isomorphism of 
$Z(\mh H)$-modules, or even of modules over the central subalgebra $\mc O (\mf t)^{W\Gamma}$. 
To work well with $HH_n (\mh H)$, we need to understand the isomorphism \eqref{eq:6} 
better and to realize it with maps induced by algebra homomorphisms.

In \cite{SolHomAHA} this is achieved (without twisting by $\natural$) with families of
representations. For every $w \in W\Gamma$ a family $\mf F_w$ of $\mh H (\mf t,W,k) \rtimes 
\Gamma$-representations parametrized by $\mf t^w$ is chosen, such that:
\begin{enumerate}[(i)]
\item in the Grothendieck group $R( \mh H (\mf t,W,k) \rtimes \Gamma)$ of finite dimensional
$\mh H (\mf t,W,k) \rtimes \Gamma$-representations, the span of $\mf F_w$ is linearly 
independent from the span of the union of the $\mf F_{w'}$ with $w'$ not conjugate to $w$ 
in $W\Gamma$, 
\item the union of all the $\mf F_w$ spans 
$\Q \otimes_\Z R (\mh H (\mf t,W,k) \times \Gamma)$.
\end{enumerate}
Each $\mf F_w$ induces an algebra homomorphism
\begin{equation}\label{eq:7}
\mc F_w : \mh H (\mf t,W,k) \rtimes \Gamma \to \mc O (\mf t^w) \otimes \End_\C (V_w) ,
\end{equation}
where $V_w$ is the vector space underlying all representations in $\mf F_w$. Recall that
by the Hochschild--Kostant--Rosenberg theorem
\[
HH_n \big( \mc O (\mf t^w) \otimes \End_\C (V_w) \big) \cong \Omega^n (\mf t^w). 
\]
It is shown in \cite{SolHomAHA} that the maps $HH_n (\mc F_w)$ together 
induce an isomorphism of $Z(\mh H (\mf t,W,k) \rtimes \Gamma)$-modules
\[
HH_n ( \mh H (\mf t,W,k) \rtimes \Gamma ) \cong \big( 
\bigoplus\nolimits_{w \in W\Gamma} \Omega^n (\mf t^w) \big)^{W\Gamma} .
\]
Unfortunately there is a problem with \cite{SolHomAHA}: the construction of the
families $\mf F_w$ does not work in general. Namely, on \cite[p. 18 and 20]{SolHomAHA} it is 
reduced to algebras of the form $\mc O (\mf t) \rtimes G$, but in that setting the families
of representations provided by \cite[(40)]{SolHomAHA} do not always satisfy (i) and (ii).
Nevertheless \cite{SolHomAHA} remains valid in most cases, because:
\begin{itemize}
\item When the parameters $k$ are of ``geometric type", one can use families of representations
from \cite{Lus1,Lus2,AMS2}.
\item For most real-valued $k$ one can use the technique with parameter deformations from
\cite[proof of Lemma 6.4]{SolHecke}, to reduce to the previous case.
\item We are not aware of any examples of graded Hecke algebras for which it is clear that they 
do not possess families of representations $\mf F_w$ with the properties (i) and (ii).
\end{itemize}
For the twisted graded Hecke algebra $\mh H = \mh H (\mf t,W,k) \rtimes \C [\Gamma,\natural]$,
the situation is less favorable: it may be impossible to find families of representations with 
the above properties. A counterexample is provided by Example \ref{ex:2.B}, which shows that for 
$\mc O (\mf t) \rtimes \C [W\Gamma,\natural]$ property (i) is problematic for representations 
with central character $W\Gamma v \in \mf t / W\Gamma$ such that $\natural$ is nontrivial in 
$H^2 ( (W\Gamma)_v, \C^\times)$.

To overcome that, we consider not only (algebraic) families of representations, but also 
families of virtual representations, in $\C \otimes_\Z R (\mh H)$.
In Lemma \ref{lem:2.12} we check that every such family canonically induces a map on Hochschild
homology, a linear combination of maps induced by algebra homomorphisms. For each $w \in W\Gamma$
we construct an algebraic family of virtual $\mh H$-representations
$\nu^1_w = \{ \nu^1_{w,v} : v \in \mf t^w \}$, which satisfies (i) and (ii).

\begin{thmintro}\label{thm:B} 
\textup{(see Theorem \ref{thm:4.4})} 
\enuma{
\item The families of virtual representations $\nu^1_w$ induce an isomorphism of vector spaces
\[
HH_n ( \mh H ) \cong \big( \bigoplus\nolimits_{w \in W\Gamma} 
\Omega^n (\mf t^w) \otimes \natural^w \big)^{W\Gamma} .
\]
\item $HH_0 ( \mh H )$ is naturally isomorphic to the set of $f$ in 
$\big( \C \otimes_\Z R (\mh H) \big)^*$ with the property: for any algebraic family
$\mf F : \lambda \mapsto \mf F_\lambda$ of $\mh H$-representations, 
$\lambda \mapsto f (\mf F_\lambda)$ is a regular function.
}
\end{thmintro}
We note that Theorem \ref{thm:B}.b is quite similar to the description of the zeroth Hochschild
homology for a reductive $p$-adic group obtained in \cite{BDK}.

The above does not yet describe the structure of $HH_n ( \mh H )$
as $\mc O (\mf t)^{W\Gamma}$-module, because in general the virtual representations $\nu^1_{w,v}$
do not admit a central character. Things can be improved by a canonical decomposition of the
category of finite dimensional tempered $\mh H$-representations, or at least the Grothendieck
group thereof (Theorem \ref{thm:4.7}). That induces a decomposition of many objects associated
to $\mh H$. In particular (Lemma \ref{lem:4.9})
\[
\nu^1_{v,w} = \sum\nolimits_{\mf d \in \Delta_{\mh H}} \nu^{1,\mf d}_{v,w} ,
\]
where each virtual $\mh H$-representation $\nu^{1,\mf d}_{v,w}$ admits 
an $\mc O (\mf t)^{W\Gamma}$-character $W\Gamma (cc(\delta) + v)$.

\begin{thmintro}\label{thm:C}
\textup{(see Corollary \ref{cor:4.12})}\\
There exists a canonical decomposition 
\[
HH_n (\mh H ) = \bigoplus\nolimits_{\mf d \in \Delta_{\mh H}} HH_n (\mh H )^{\mf d}
\]
such that the injection
\[
HH_n (\mh H )^{\mf d} \to \Big( \bigoplus\nolimits_{w \in W \Gamma} 
\Omega^n (\mf t^w) \otimes \natural^w \Big)^{W \Gamma}
\]
from Theorem \ref{thm:B} becomes $\mc O (\mf t)^{W \Gamma}$-linear if we let 
$\mc O (\mf t)^{W \Gamma}$ act at $(w,v)$ via evaluation at $W \Gamma (cc (\delta) + v)$.
\end{thmintro}

\textbf{Applications to $p$-adic groups.}\\
It has been known for a long time that affine Hecke algebras play a role in the representation 
theory of reductive $p$-adic groups. The author has made that precise in full generality 
\cite{SolEnd}, although it turned out that (twisted) graded Hecke algebras are involved more
naturally. This has opened several new research options, for instance, it can be used to 
determine homologies of the Hecke algebra $\mc H (G)$ of an arbitrary reductive $p$-adic group 
$G$. Locally $\mc H (G)$ is Morita equivalent with (a local part of) a twisted graded Hecke
algebra \cite[\S 7--8]{SolEnd}. That prompted us to compute the Hochschild homology of such 
algebras. In Proposition \ref{prop:4.13} we show that there is a canonical isomorphism of
vector spaces
\[
HH_n \big( \mh H (\mf t,W,\natural) \rtimes \C[\Gamma,\natural] \big) \cong 
HH_n \big( \mc O (\mf t) \rtimes \C[W\Gamma,\natural] \big) .
\]
We expect that similarly $HH_n (\mc H (G))$ will be isomorphic to a direct sum of terms
\[
HH_n \big( \mc O (T_{\mf s}) \rtimes \C [W_{\mf s},\natural_{\mf s}] \big) .
\]
Here the complex torus $T_{\mf s}$, the finite group $W_{\mf s}$ and the 2-cocycle 
$\natural_{\mf s}$ are canonically associated to a Berstein component $\Irr (G)^{\mf s}$ of
$\Irr (G)$. The details will appear in a forthcoming paper.

One substantial complication is that we will have to deal with discontinuous families of
twisted graded Hecke algebras. To the end we will employ at least two strategies:
\begin{itemize}
\item realize Hochschild homology in terms of algebraic families of representations,
\item make such families less discontinuous in the sens that at least the involved vector
space $\mf t$ and the finite group $W\Gamma$ are locally constant.
\end{itemize}
The former already permeates this paper, for the latter we can replace the twisted graded
Hecke algebras by larger Morita equivalent algebras. In Paragraph \ref{par:Mor} we generalize 
our main results to such algebras, which combine features of twisted crossed products with 
commutative algebras and of graded Hecke algebras.\\

\textbf{Structure of the paper.}\\
In Paragraph \ref{par:HHcrossed} we introduce the characters $\natural^g$ and we prove
Theorem \ref{thm:A}. Some generalities involving families of representations, valid for many
algebras, are discussed in Paragraph \ref{par:famCrossed}. To prepare for Theorem 
\ref{thm:B}, we establish a simpler analogue with algebras of the form $\mc O (V) \rtimes
\C [G,\natural]$, in Paragraph \ref{par:realiz}. As an intermediate step we map
$HH_n (\mc O (V) \rtimes \C [G,\natural])$ to $n$-forms on some algebraic varieties, via
families of representations.

We start Section \ref{sec:HHGHA} with recalling the definition of a (twisted) graded Hecke
algebra. Then we generalize some representation theoretic results, which allow to reduce
certain issues for $\mh H (\mf t,W,k) \rtimes \C[\Gamma,\natural]$ to its version
$\mc O (\mf t) \rtimes \C [W\Gamma,\natural]$ with $k=0$. Next we prove Theorem \ref{thm:B}, 
in many small steps. Again it goes via differential forms coming from auxiliary algebraic 
families of representations. After that we wrap up the proof of Theorem \ref{thm:C}.\\[2mm]

\textbf{Acknowledgements.}\\
We thank David Kazhdan and Roman Bezrukavnikov for some inspiring email conversations
on related topics.

\renewcommand{\theequation}{\arabic{section}.\arabic{equation}}
\numberwithin{equation}{section}

\section{Twisted crossed product algebras}
\label{sec:crossed}

\subsection{Hochschild homology via differential forms} \ 
\label{par:HHcrossed}

Let $G$ be a finite group, let $\natural : G \times G \to \C^\times$ be a 2-cocycle and 
form the twisted group algebra $\C [G,\natural]$. It has a basis $\{ T_g : g \in G \}$ 
and multiplication rules
\[
T_g \cdot T_{g'} = \natural (g,g') T_{g g'} .
\]
By the theory of Schur multipliers \cite[\S 53]{CuRe} there exists a finite central extension
\begin{equation}\label{eq:2.20}
1 \to Z \to \tilde G \to G \to 1
\end{equation}
such that the corresponding lift of $\natural$ is trivial in $H^2 (\tilde G,\C^\times)$.
Then there exists a minimal central idempotent $p_\natural \in \C [Z]$ and an algebra isomorphism
\begin{equation}\label{eq:2.1}
\begin{array}{ccc}
p_\natural \C [\tilde G] & \to & \C [G,\natural] \\
p_\natural \tilde g & \mapsto & c_{\tilde g} T_g 
\end{array}.
\end{equation}
Here $\tilde g \in \tilde G$ has image $g$ in $G$, and $c_{\tilde g} \in \C$ is a suitable scalar.
Notice that $p_\natural \C [\tilde G]$ is a direct summand of the semisimple algebra $\C [\tilde G]$,
so itself semisimple. The minimal idempotent $p_\natural \in \C [Z]$ is associated to some 
character $\chi_\natural$ of $Z$, so 
\[
p_\natural = |Z|^{-1} \sum\nolimits_{z \in Z} \chi_\natural^{-1}(z) z.
\]
Let $\natural^g$ be the character
\[
\begin{array}{ccl}
Z_{\tilde G} (g) & \to & \C^\times \\
\tilde h & \mapsto & \chi_\natural^{-1} ([\tilde g,\tilde h]) = \chi_\natural ([\tilde h,\tilde g])
\end{array} .
\]
Here $\tilde g \in \tilde G$ is a lift of $g \in G$, and the choice does not matter because any
other lift differs from $\tilde g$ by a central element. The kernel of $\natural^g$ contains $Z$, 
so we can also regard it as a character of $Z_G (g)$. As such, one can also express it as
\begin{equation}\label{eq:2.31}
\natural^g (h) = T_g T_h T_g^{-1} T_h^{-1} = T_h T_g T_h^{-1} T_g^{-1} .
\end{equation}
This shows that $\natural^g$ is insensitive to rescaling $T_g$ and $T_h$, which entails that
$\natural^g$ depends only on $g$ and the cohomology class of $\natural$.

Recall that $HH_0 (A)^*$ is the space of trace functions on an algebra $A$.

\begin{lem}\label{lem:2.3}
\enuma{
\item For $g \in G$ with $\natural^g = 1$, there exists a unique trace function $\nu_g$
on $\C [G,\natural]$ with $\nu_g (T_g) = 1$ and $\nu_g (T_{g'}) = 0$ if $g$ and $g'$ are not
conjugate in $G$.
\item Let $\langle G \rangle$ be a set of representatives for the conjugacy classes in $G$.
The set $\{ \nu_g : g \in \langle G \rangle, \natural^g = 1\}$ is a basis of 
$HH_0 (\C [G,\natural])^*$. The number of inequivalent irreducible representation of 
$\C[G,\natural]$ equals $| \{ g \in \langle G \rangle, \natural^g = 1\} |$.
}
\end{lem}
\begin{proof}
(a) Since $\C [G,\natural] \cong p_\natural \C [\tilde G]$ is a direct summand of $\C [\tilde G]$,
every trace function on $p_\natural \C [\tilde G]$ can be extended to one on $\C [\tilde G]$.
A basis of $HH_0 (\C [\tilde G])^*$ is the set of indicator functions
$1_{\tilde C}$ for the conjugacy classes $\tilde C$ in $\tilde G$. 

Suppose that $\tilde g \in \tilde C$ and $\natural^g \neq 1$ (where $g$ is the image of
$\tilde g$ in $G$). Then $\tilde g$ is $\tilde G$-conjugate to $\tilde g z$ for some $z \in Z$
with $\chi_\natural (z) \neq 1$. Hence $1_{\tilde C} (p_\natural T_{\tilde g})$ is a multiple 
of $\sum_{n=1}^{\mr{ord}(z)} \chi_\natural (z^n) = 0$, which implies that $1_{\tilde C}$
vanishes on $p_\natural \C [\tilde G]$. 

On the other hand, suppose that $\tilde g \in \tilde C$ an $\natural^g = 1$. Then
$1_{\tilde C} \big|_{p_\natural \C [\tilde G]}$ is nonzero and has support 
\[
p_\natural \, \mr{span} \{ T_{\tilde g} : \tilde g \in \tilde C\} =
p_\natural \, \mr{span} \{ T_{\tilde g z} : \tilde g \in \tilde C, z \in Z \} . 
\]
Thus $1_{\tilde C}$ defines a trace function on $\C [G,\natural] \cong p_\natural \C [\tilde G]$
supported on the conjugacy class of $g$ in $G$. A unique scalar multiple $\nu_g$ of $1_{\tilde C}$
has $\nu_g (T_g) = 1$.\\
(b) The above argument also shows that $\{ \nu_g : g \in G\}$ spans $HH_0 (\C[G,\natural])^*$. 
If we pick just one $g$ from every conjugacy class, the span does not change
and the set becomes linearly independent, so a basis. 

As the algebra $\C[G,\natural]$ is semisimple, it number of irreducible representations equals
the dimension of $HH^0 (\C [G,\natural])$.
\end{proof}

Let $V$ be a nonsingular affine variety over $\C$, with an algebraic $G$-action. Then $G$
also acts on the algebra of regular functions $\mc O (V)$. We want to compute the Hochschild 
homology of $\mc O (V) \rtimes \C [G,\natural]$ (as defined in the introduction).
Let $\langle G \rangle \subset G$ be a set of representatives for the conjugacy classes in $G$. 
We denote the set of (algebraic) differential $n$-forms on $V$ by $\Omega^n (V)$.

\begin{thm}\label{thm:2.1}
There exists an isomorphism of $\mc O (V)^G$-modules
\[
HH_n \big( \mc O (V) \rtimes \C [G,\natural] \big) \cong \bigoplus\nolimits_{g \in \langle G 
\rangle} \big( \Omega^n (V^g) \otimes \natural^g \big)^{Z_G (g)} .
\]
\end{thm}

Important background for all our computations is the Hochschild--Kostant--Rosen\-berg
(HKR) theorem, which provides a natural isomorphism of $\mc O (V)$-modules
\begin{equation}\label{eq:2.21}
HH_n (\mc O (V)) \cong \Omega^n (V) .
\end{equation}
Recall \cite[\S 1.1]{Lod} that the Hochschild homology of a unital algebra can be computed as 
the homology of the bar complex $(C_* (A), b_*)$, where $C_* (A) = A^{\otimes (n+1)}$ and
\[
b_n (a_0 \otimes \cdots \otimes a_n) = \sum_{i=0}^{n-1} a_0 \otimes \cdots \otimes a_i a_{i+1} \otimes
\cdots \otimes a_n + (-1)^n a_n a_0 \otimes a_1 \otimes \cdots \otimes a_{n-1} .
\]
The isomorphism \eqref{eq:2.21} is realized by the map
\begin{equation}\label{eq:2.22}
\begin{array}{cccc}
\Omega : & C_n (\mc O (V^P)) & \to & \Omega^n (\mc O (V^P)) \\
& f_0 \otimes f_1 \otimes \cdots \otimes f_n & \mapsto & f_0 \textup{d}f_1 \cdots \textup{d}f_n / n!
\end{array}.
\end{equation}
We let $\tilde G$ act on $V$ via its quotient $G$. Then \eqref{eq:2.1} induces an algebra isomorphism
\[
\mc O (V) \rtimes \C [G,\natural] \cong \mc O (V) \rtimes p_\natural \C [\tilde G] =
p_\natural (\mc O (V) \rtimes \tilde G) ,
\]
where the right hand side is a direct summand of the crossed product algebra 
$\mc O (V) \rtimes \tilde G$. Brylinski \cite{Bry} and Nistor \cite{Nis} computed the Hochschild
homology of such algebras. It can be done in three steps:
\begin{itemize}
\item For each $\tilde g \in \tilde G$, $C_* (\mc O (V) \rtimes \tilde G)$ contains a subcomplex
$\tilde g C_* (\mc O (V))$. It can be identified with the complex that computes the Hochschild
homology of $\mc O (V)$ with coefficients in the bimodule $\tilde g \mc O (V)$, so
\begin{equation}\label{eq:2.4}
H_n \big( \tilde g C_* (\mc O (V)), b_* \big) = HH_n \big( \mc O (V), \tilde g \mc O (V) \big) .
\end{equation}
\item Varying on the HKR theorem, one computes that
\begin{equation}\label{eq:2.5}
HH_n \big( \mc O (V), \tilde g \mc O (V) \big) \cong \Omega^n (M^{\tilde g} )^{Z_{\tilde G} (\tilde g)} .
\end{equation}
On the level of Hochschild complexes, the isomorphism comes from the map \eqref{eq:2.22} followed
by averaging over $Z_G (g)$.
\item Let $\langle \tilde G \rangle$ be a set of representatives for the conjugacy classes in
$\tilde G$. One shows that the inclusion
\[
\bigoplus\nolimits_{\tilde g \in \langle \tilde G \rangle} \tilde g C_* (\mc O (V)) \to
C_* (\mc O (V) \rtimes \tilde G)
\]
is a quasi-isomorphism.
\end{itemize}
Consequently one obtains an isomorphism of $\mc O (V)^{\tilde G}$-modules 
\begin{equation}\label{eq:2.2}
HH_n \big( \mc O (V) \rtimes \tilde G \big) \cong \bigoplus\nolimits_{\tilde g \in \langle \tilde G 
\rangle} \Omega^n (V^{\tilde g} )^{Z_{\tilde G} (\tilde g)} \cong 
\Big( \bigoplus\nolimits_{\tilde g \in \tilde G} \Omega^n (V^{\tilde g} ) \Big)^G ,
\end{equation}
which is made natural in \cite[Theorem 2.11]{Nis}.

While \eqref{eq:2.2} holds for any smooth action of a finite group on a manifold, our setting is more
specific, with a central subgroup $Z$ that acts trivially. Hence $\Omega^n (V^{\tilde g})^{Z_{\tilde G}
(\tilde g)}$ depends only the image $g$ of $\tilde g$ in $G$:
\[
\Omega^n (V^{\tilde g})^{Z_{\tilde G} (\tilde g)} = \Omega^n (V^g)^{Z_G (\tilde g)} .
\]
Notice that $Z_G (\tilde g)$ makes sense because the conjugation action of $\tilde G$ on itself 
factors through an action of $G$ on $\tilde G$. In general $Z_G (\tilde g)$ is contained in 
$Z_G (g)$, but they need not be equal. With this at hand, \eqref{eq:2.2} specializes to an isomorphism
of $\mc O (V)$-modules
\begin{equation}\label{eq:2.3}
HH_n \big( \mc O (V) \rtimes \tilde G \big) \cong 
\Big( \bigoplus\nolimits_{\tilde g Z \in G} \bigoplus\nolimits_{z \in Z} \Omega^n (M^g ) \Big)^G . 
\end{equation}
The difference between the various direct summands $\Omega^n (M^g)$ is that they come from distinct
subcomplexes $\tilde g z C_* (\mc O (V))$.\\

\noindent \emph{Proof of Theorem \ref{thm:2.1}}\\
Since $HH_n (A)$ is always a module over $Z(A)$ and $p_\natural$ is a central idempotent:
\begin{equation}\label{eq:2.6}
HH_n \big( \mc O (V) \rtimes \C [G,\natural] \big) \cong
HH_n \big( p_\natural (\mc O (V) \rtimes \C [G,\natural] ) \big) =
p_\natural HH_n \big( \mc O (V) \rtimes \C [G,\natural] \big) .
\end{equation}
By \eqref{eq:2.3}, \eqref{eq:2.4} and \eqref{eq:2.5} the right hand side of \eqref{eq:2.6} equals
\[
\Big( p_\natural \bigoplus\nolimits_{z \in Z} H_n \big( \tilde g C_* (\mc O (V)), b_* \big) \Big)^G .
\]
This expression decomposes naturally as a direct sum over the conjugacy classes of $G$, namely
\begin{equation}\label{eq:2.7}
HH_n \big( \mc O (V) \rtimes \C [G,\natural] \big) \cong \bigoplus\nolimits_{\tilde g Z \in \langle G 
\rangle} p_\natural \Big( \bigoplus\nolimits_{z \in Z} 
H_n \big( \tilde g C_* (\mc O (V)), b_* \big) \Big)^{Z_G (g)} .
\end{equation}
The action of $h \in Z_G (g)$ (with a lift $\tilde h \in \tilde G$) sends 
$\tilde g z \cdot c \in \tilde g z C_* (\mc O (V))$ to
\[
\tilde h \tilde g {\tilde h}^{-1} z \cdot h(c) = \tilde g [\tilde h,\tilde g] z \cdot h (c) .
\]
We find
\begin{multline}\label{eq:2.8}
\Big( \bigoplus\nolimits_{z \in Z} H_n \big( \tilde g C_* (\mc O (V)), b_* \big) \Big)^{Z_G (g)} = \\
\Big\{ \sum\nolimits_{z \in Z} \omega_z \in \bigoplus\nolimits_{z \in Z} \Omega^n (\mc O (V^g)) :
\omega_{[\tilde h, \tilde g] z} = h (\omega_z) \; \forall h \in Z_G (g) \Big\} .
\end{multline}
The shape of $p_\natural$ entails that its image in \eqref{eq:2.8} is
\begin{multline}\label{eq:2.9} 
\Big\{ \sum\nolimits_{z \in Z} \omega_z \in \bigoplus\nolimits_{z \in Z} \Omega^n (\mc O (V^g)) : \\
h (\omega_{[\tilde g,\tilde h] z}) = \omega_z \; \forall h \in Z_G (g), \omega_{z' z} = 
\chi_\natural^{-1}(z') \omega_z \; \forall z' \in Z \Big\} .
\end{multline}
The two conditions in \eqref{eq:2.9} are equivalent with
\begin{equation}\label{eq:2.10}
\omega_{z' z} = \chi_\natural^{-1}(z') \omega_z \text{ and } 
\omega_z = \natural^g (h) h(\omega_z) \quad \forall z,z' \in Z, h \in Z_G (g) .
\end{equation}
From \eqref{eq:2.7}--\eqref{eq:2.10} we obtain the required description of 
$HH_n \big( \mc O (V) \rtimes \C [G,\natural] \big)$ as $\mc O (V)^G$-module. $\qquad \Box$\\

Unfortunately this isomorphism does not seem to be natural, unlike \eqref{eq:2.2}. In the way
we constructed it, it depends on a choice of representatives in $\tilde G$ for the conjugacy
classes of $G$ and the choice of the algebra isomorphism \eqref{eq:2.1}. This can be improved
a little by a more explicit construction, like in the case without a 2-cocycle. Indeed, like
over there we can compose \eqref{eq:2.22} with averaging over $Z_G (g)$. But here conjugation
by $h \in Z_G (g)$ is a bit more subtle, namely
\begin{align*}
T_h (T_g f_0 \otimes f_1 \otimes \cdots \otimes f_n) T_h^{-1} & =
T_h T_g T_h ^{-1} \, h \cdot (f_0 \otimes f_1 \otimes \cdots \otimes f_n) \\
& = \natural^g (h) T_g \, h \cdot (f_0 \otimes f_1 \otimes \cdots \otimes f_n) .
\end{align*}
Hence we can realize Theorem \ref{thm:2.1} for the summand indexed by $g$ as
\begin{equation}\label{eq:2.23}
\begin{array}{ccc}
T_g C_n (\mc O (V)) & \to & ( \Omega^n (V^g) \otimes \natural^g )^{Z_G (g)} \\
T_g f_0 \otimes f_1 \otimes \cdots \otimes f_n & \mapsto & {\displaystyle \sum_{h \in Z_G (g)} 
\frac{\natural^g (h) \, h \cdot (f_0 \textup{d} f_1 \cdots \textup{d} f_n )}{|Z_G (g)| \, n!} }
\end{array}.
\end{equation}
This entails that on the summand indexed by $g$, the isomorphism from Theorem \ref{thm:2.1}
is canonical up to a scalar (from the choice of $T_g$). To analyse the dependence on the choice
of representatives of the conjugacy classes, we define
\[
\natural^g (h) = T_h T_g T_h^{-1} T_{h g h^{-1}}^{-1} \in \C^\times
\qquad \text{for all } g,h \in G.
\]

\begin{lem}\label{lem:2.10}
Let $g, h, \tilde h \in G$. 
\enuma{
\item $\natural^g (\tilde h h) = \natural^{h g h^{-1}}(\tilde h) \natural^g (h)$.
\item $\natural^g (h) : (\natural^g,\C) \to (h^{-1} \cdot \natural^{h g h^{-1}}, \C)$
is an isomorphism of $Z_G (g)$-representations.
}
\end{lem}
\begin{proof}
(a) Using $\natural^g (h) = T_h^{-1} T_{h g h^{-1}}^{-1} T_h T_g$ we compute
\begin{align*}
\natural^{h g h^{-1}} (\tilde h) \natural^g (h) T_g^{-1} & = 
\natural^{h g h^{-1}} (\tilde h) T_h^{-1} T_{h g h^{-1}}^{-1} T_h \\
& = T_h^{-1} \natural^{h g h^{-1}} (\tilde h) T_{h g h^{-1}}^{-1} T_h \\
& = T_h^{-1} T_{\tilde h}^{-1} T_{\tilde h h g h^{-1} \tilde h^{-1}}^{-1} T_{\tilde h} T_h \\
& = (\natural (\tilde h,h) T_{h \tilde h})^{-1} T_{\tilde h h g h^{-1} \tilde h^{-1}}^{-1}
(\natural (\tilde h,h) T_{h \tilde h}) \\
& = T_{h \tilde h}^{-1} T_{\tilde h h g h^{-1} \tilde h^{-1}}^{-1} T_{h \tilde h} 
\quad = \quad \natural^g (\tilde h h) T_g^{-1} .
\end{align*}
(b) Assume that $\tilde h \in Z_G (h g h^{-1})$. Applying part (a) twice, we find
\[
\natural^{h g h^{-1}} (\tilde h) \natural^g (h) = \natural^g (\tilde h h) =
\natural^g (h) \natural^g (h^{-1} \tilde h h) .
\]
Hence $\natural^g (h)$ intertwines the $Z_G (h g h^{-1})$-representations
$h \cdot \natural^g$ and $\natural^{h g h^{-1}}$.
\end{proof}

Lemma \ref{lem:2.10} provides a canonical bijection
\begin{equation}\label{eq:2.32}
\Omega^n (h^{-1}) \otimes \natural^g (h) : \Omega^n ( V^g ) \otimes \natural^g 
\to \Omega^n ( V^{h g h^{-1}}) \otimes \natural^{h g h^{-1}} , 
\end{equation}
which intertwines the $Z_G (g)$-actions (where $Z_G (g)$ acts on the right hand side
via precomposing with conjugation by $h$). Regarding \eqref{eq:2.32} as an action of 
$h \in G$ on the sum of these spaces, we can reformulate Theorem \ref{thm:2.1} as
\begin{equation}\label{eq:2.33}
HH_n \big( \mc O (V) \rtimes \C [G,\natural] \big) \cong \bigoplus_{g \in \langle G \rangle}
\big( \Omega^n (V^g) \otimes \natural^g \big)^{Z_G (g)} 
\cong \Big( \bigoplus_{g \in G} \Omega^n (V^g) \otimes \natural^g \Big)^G .
\end{equation}
Consider any
\[
\omega = \sum\nolimits_{g \in G} T_g \omega_g \in 
\Big( \bigoplus\nolimits_{g \in G} \Omega^n (V^g) \otimes \natural^g \Big)^G .
\]
By construction
\begin{equation}\label{eq:2.34}
T_{h g h^{-1}} \omega_{h g h^{-1}} = T_h T_g \omega_g T_h^{-1} =
T_h T_g T_h^{-1} h \cdot \omega_g = \natural^g (h) T_{h g h^{-1}} h \cdot \omega_g .
\end{equation}
We deduce that $\omega_{h g h^{-1}} = \natural^g (h) h \cdot \omega_g$.

Next we relate Theorem \ref{thm:2.1} to $\Irr \big( \mc O (V) \rtimes \C [G,\natural] \big)$. 
For $g \in G$ and $v \in V^g$ with $\natural^g |_{G_v \cap Z_G (g)} = 1$,
we define $\nu_{g,v} \in HH_0 (\mc O (V) \rtimes \C [G,\natural])^*$ as evaluation at 
$(g,v)$ in the expression 
\begin{equation}\label{eq:2.11}
HH_0 ( \mc O (V) \rtimes \C [G,\natural] ) \cong 
\big( \bigoplus\nolimits_{g \in G} \mc O (V^g) \otimes \natural^g \big)^G
\end{equation}
from Theorem \eqref{eq:2.33}. By \eqref{eq:2.34}, for any $h \in G$:
\begin{equation}\label{eq:2.16}
\nu_{h g h^{-1},hv}(\omega) = \nu_{h g h^{-1},hv}(\natural^g (h) \, h \cdot \omega) =
\natural^g (h) \nu_{g,v}(\omega) .
\end{equation}
Hence $\nu_{hgh^{-1},hv} = \natural^g (h) \nu_{g,v}$. From \eqref{eq:2.23} we also see how 
$\nu_{g,v}$ becomes a map $\mc O (V) \rtimes \C [G,\natural] \to \C$ supported on $T_g \mc O (V)$. 

\begin{lem}\label{lem:2.4}
The following numbers are equal:
\begin{itemize}
\item[(i)] the number of inequivalent irreducible representations of $\mc O (V) \rtimes
\C [G,\natural]$ with $\mc O (V)^G$-character $Gv$,
\item[(ii)] $| \{ g \in \langle G_v \rangle : \natural^g |_{G_v \cap Z_G (g)} = 1 \}|$,
where $\langle G_v \rangle$ is a set of representatives for the conjugacy classes in $G_v$,
\item[(iii)] the dimension of the specialization of $HH_0 ( \mc O (V) \rtimes
\C [G,\natural] )$ at $Gv$,
\item[(iv)] $|\{ \nu_{g,v'} : g \in \langle G \rangle, v' \in (V^g \cap Gv ) / Z_G (g),
\natural^g |_{G_v \cap Z_G (g)} = 1 \} |$. 
\end{itemize}
In (iv) $v' \in (V^g \cap Gv ) / Z_G (g)$ means that from every $Z_G (g)$-orbit we pick
one element in $V^g \cap Gv$.
\end{lem}
\begin{proof}
By Mackey theory there is a bijection from $\Irr (\C [G_v,\natural])$ to the set in (i), namely
\begin{equation}\label{eq:2.24}
\rho \mapsto \mr{ind}_{\mc O (V) \rtimes \C [G,\natural]}^{\mc O (V) \rtimes
\C [G_v,\natural]} (\C_v \otimes \rho) .
\end{equation}
By Lemma \ref{lem:2.3}, $|\Irr (\C [G_v,\natural])|$ equals (ii).

From Theorem \ref{thm:2.1} we see that specializing $HH_0 ( \mc O (V) \rtimes
\C [G,\natural] )$ at $Gv$ yields
\[
\Big( \bigoplus\nolimits_{g \in G} C (V^g \cap Gv) \otimes \natural^g \Big)^G \cong
\Big( \bigoplus\nolimits_{g \in G_v} C(\{v\}) \otimes \natural^g \Big)^{G_v} .
\]
The dimension of the right hand side is (ii) and the dimension of the left hand side
equals (iv).
\end{proof}

As a consequence of Lemma \ref{lem:2.4}, we record that there is a bijection 
\begin{equation}\label{eq:2.12}
\Irr (\mc O (V) \rtimes \C [G,\natural]) \; \longleftrightarrow \;
\{ \nu_{g,v} :  g \in \langle G \rangle, v \in V^g / Z_G (g), 
\natural^g |_{G_v \cap Z_G (g)} = 1\} ,
\end{equation}
which preserves the underlying $G$-orbits in $V$. 

\begin{ex}\label{ex:2.A}
We illustrate the constructions in this section with an example that exhibits some 
non-standard behaviour. Let $Q_8$ be the quaternion group, with centre $Z(Q_8) = \{ 1, -1\}$.
Let $G$ be the quotient 
\[
Q_8 / Z(Q_8) = \{ \pm \mb{1}, \pm \mb{i}, \pm \mb{j}, \pm \mb{k}\} .
\]
For a nontrivial 2-cocycle on $G$, let $\chi_\natural$ be the nontrivial character of
$Z(Q_8)$ and define $\C [G,\natural] = p_\natural \C [Q_8]$.
From calculations like $T_{\pm \mb j} T_{\pm \mb i} T_{\pm \mb j}^{-1} = - T_{\pm \mb i}$ we obtain
\[
\natural^{\pm \mb i}(g) = \left\{
\begin{array}{cc}
1 & g \in \langle \pm \mb{i} \rangle \\
-1 & g \notin \langle \pm \mb{i} \rangle 
\end{array}
\right. .
\]
Similarly $\natural^{\pm \mb j}$ and $\natural^{\pm \mb k}$ are nontrivial characters
of $G$, while $\natural^{\pm \mb 1} = \mr{triv}_G$.

The group $G$ acts on $V = \C^2$ by 
\[
\pm \mb{i} \cdot (z_1,z_2) = (-z_1,z_2) ,\quad
\pm \mb{j} \cdot (z_1,z_2) = (z_1,-z_2) ,\quad
\pm \mb{k} \cdot (z_1,z_2) = (-z_1,-z_2) .
\]
From Lemma \ref{lem:2.4} we can compute the number of irreducible representations of
$\mc O (V) \rtimes \C [G,\natural]$ with a fixed $\mc O (V)^G$-character:
\[
\begin{array}{c|c|c|c|c}
(z_1,z_2) \in \C^2 / G & z_1 \neq 0 \neq z_2 & z_1 = 0 \neq z_2 & z_1 \neq 0 = z_2 & z_1 = 0 = z_2 \\
\hline
\# \text{ irreps} & 1 & 2 & 2 & 1 
\end{array}
\]
We work out the description of $HH_n (\mc O (V) \rtimes \C [G,\natural])$ from Theorem \ref{thm:2.1}:
\begin{itemize}
\item $g = \pm \mb{1}$: $V^g = V, (\Omega^n (V^g) \otimes \natural^g)^{Z_G (g)} = \Omega^n (V)^G$,
\item $g = \pm \mb{i}$: $V^g = \{0\} \times \C, (\Omega^0 (V^g) \otimes \natural^g)^{Z_G (g)} =
\{ f \in \mc O (V^g) : f(-x) = -f(x) \}$, 
$(\Omega^1 (V^g) \otimes \natural^g)^{Z_G (g)} =
\{ f \textup{d} z_2 \in \Omega^1 (V^g) : f(-x) = f(x) \}$,
\item $g = \pm \mb{j}$: $V^g = \C \times \{0\}, (\Omega^0 (V^g) \otimes \natural^g)^{Z_G (g)} =
\{ f \in \mc O (V^g) : f(-x) = -f(x) \}$, 
$(\Omega^1 (V^g) \otimes \natural^g)^{Z_G (g)} =
\{ f \textup{d} z_1 \in \Omega^1 (V^g) : f(-x) = f(x) \}$,
\item $g = \pm \mb{k}$: $V^g = \{(0,0)\}$, $(\Omega^0 (V^g) \otimes \natural^g)^{Z_G (g)} = 0$.
\end{itemize}
\end{ex}

\subsection{Algebraic families of representations} \ 
\label{par:famCrossed}

Let $A$ be a $\C$-algebra and let $\mr{Mod}_f (A)$ be the category of finite length $A$-modules.
Let $R(A)$ be the Grothendieck group of $\mr{Mod}_f (A)$. Assume that, for every $a \in A$ and
every $\pi \in \mr{Mod}_f (A)$, $\mr{tr}\, \pi (a)$ is well-defined. This holds for instance 
if all finite length $A$-representations have finite dimension. That is the case for all algebras 
that we study in this paper, because they have finite rank as modules over their centre.
Under this condition, there is a natural pairing
\begin{equation}\label{eq:2.42}
\begin{array}{ccc}
HH_0 (A) \times \mr{Mod}_f (A) & \to & \C \\
(a,\pi) & \mapsto & \mr{tr} \, \pi (a)
\end{array},
\end{equation}
which induces a $\C$-bilinear map
\[
HH_0 (A) \times \C \otimes_\Z R(A) \to \C .
\]
These can also be interpreted as natural linear maps
\begin{align}
\label{eq:2.14} & R(A) \to \C \otimes_\Z R(A) \to HH_0 (A)^* ,\\
\label{eq:2.43} & HH_0 (A) \to (\C \otimes_\Z R(A))^* .
\end{align}
By the linear independence of irreducible characters \eqref{eq:2.14} is injective, so this 
identifies $R(A)$ and $\C \otimes_\Z R(A)$ with subgroups of $HH_0 (A)^*$. 

By an algebraic family $\mf F$ of $A$-representations over a complex affine  variety $Y$ we mean 
a family of $A$-representations $\mf F_y \; (y \in Y)$, all on the same finite dimensional vector 
space $W$, which together give an algebra homomorphism
\[
\begin{array}{cccc}
\mc F_Y : & A & \to & \mc O (Y) \otimes \End_\C (W) \\
& a & \mapsto & [ y \mapsto \mf F_y (a) ] 
\end{array}.
\]
For any $a \in A$, the map
\[
Y \to \C : y \mapsto \mr{tr}\, \mf F_y (a)
\]
is a regular function. We call a linear function $f$ on $\C \otimes_\Z R(A)$
regular if for every algebraic family of $A$-representations $\mf F_Y$, 
\[
\text{the function } Y \to \C : y \mapsto f (\mf F_y) \text{ is regular.}
\]
Then the image of $HH_0 (A)$ under \eqref{eq:2.43} is contained in 
\begin{equation}\label{eq:2.44}
(\C \otimes_\Z R(A))^*_{\mr{reg}} = \{ f \in (\C \otimes_\Z R(A))^* : f \text{ is regular} \} .
\end{equation}
Assume that $Y$ is nonsingular. By the Hochschild--Kostant--Rosenberg theorem and Morita 
invariance, the Hochschild homology of $\mc O (Y) \otimes \End_\C (W)$ is $\Omega^* (Y)$. 
We recall from \cite[\S 1.2]{Lod} that the isomorphism
\[
HH_n \big( \mc O (Y) \otimes \mr{End}_\C (W) \big) \to HH_n (\mc O (Y))
\]
can be implemented by the generalized trace map
\[
\begin{array}{cccc}
\mr{gtr} : & C_n (\mc O (Y) \otimes \End_\C (W)) & \to & C_n (\mc O (Y)) \\
& f_0 m_0 \otimes f_1 m_1 \otimes \cdots \otimes f_n m_n & \mapsto &
\mr{tr}(m_0 m_1 \cdots m_n) f_0 \otimes f_1 \otimes \cdots \otimes f_n
\end{array},
\]
where $f_i \in \mc O (Y)$ and $m_i \in \mr{End}_\C (W)$.

\begin{lem}\label{lem:2.2}
\enuma{
\item For $y \in Y$, the map $\mr{ev}_y \circ \mr{gtr} \circ C_* (\mc F_Y)$ depends only on
the semisimplification of the $A$-representation $\mf F_y$. 
\item The maps $\mr{gtr} \circ 
C_* (\mc F_Y)$ and 
\[
HH_n (\mc F_Y) : HH_n (A) \to \Omega^n (Y)
\] 
depend only on the image of the family $\mf F$ in $R (A)$.
}
\end{lem}
\begin{proof}
(a) Let $0 = W_0 \subset W_1 \subset \cdots \subset W_k = W$ be a composition series of the 
$A$-representation $\mf F_y$. If $m \in \mf F_y (A)$ maps every $W_j$ to $W_{j-1}$, then 
tr$(m m_1 \cdots m_n) = 0$ for all $m_i \in \mf F_y (A)$.
Hence $\mr{ev}_y \circ \mr{gtr} \circ C_* (\mc F_Y)$ factors through
\[
C_* \big( \bigoplus\nolimits_j \mr{End}_\C (W_j / W_{j-1})) \big),
\] 
and can be computed from the semisimplification $\bigoplus_j \mr{End}_\C (W_j / W_{j-1})$ of 
$(\mf F_y,W)$. \\
(b) We recall that the HKR isomorphism is realized by the map $\Omega$ from 
\eqref{eq:2.22}. Hence $HH_n (\mc F_Y)$ with target $\Omega^n (Y)$ can be realized as 
$\Omega \circ \mr{gtr} \circ C_* (\mc F_Y)$. Combine that with the first claim.
\end{proof}

We will often use a generalization of of Lemma \ref{lem:2.2} to virtual representations:

\begin{lem}\label{lem:2.12}
Let $\mf F_i$ be a finite collection of algebraic families of $A$-representations over $Y$.
For any $\lambda_i \in \C$ there is a well-defined map
\[
\sum\nolimits_i \lambda_i HH_n (\mc F_i) : HH_n (A) \to \Omega^n (Y) .
\]
If $\mf F'_j, \lambda_j$ are data of the same kind and 
\[
\sum\nolimits_i \lambda_i \mf F_{i,y} = \sum\nolimits_j \lambda'_j \mf F'_{j,y}
\quad \text{ in } \C \otimes_\Z R (A), \text{ for all } y \in Y,
\]
then $\sum_i \lambda_i HH_n (\mc F_i) = \sum_j \lambda_j HH_n (\mc F'_j)$.
\end{lem}
\begin{proof}
All the $\lambda_i$ and the $\lambda'_j$ live in one finitely generated subgroup of $\C$,
which we can express as $\bigoplus_{b \in B} \Z b$. Accordingly we write
\[
\lambda_i = \sum\nolimits_{b \in B} \lambda_{i,b} b ,\; \lambda'_j = \sum\nolimits_{b \in B} 
\lambda'_{j,b} b \text{ with } \lambda_{i,b}, \lambda'_{j,b} \in \Z.
\]
Now we have to show that 
\begin{equation}\label{eq:2.41}
\sum\nolimits_i \lambda_{i,b} HH_n (\mc F_i) = \sum\nolimits_j \lambda_{j,b} HH_n (\mc F'_j)
\end{equation}
for every $b \in B$. We note that, by the $\Z$-linear independence of $B$:
\begin{equation}\label{eq:2.19}
\sum\nolimits_i \lambda_{i,b} \mf F_{i,y} = \sum\nolimits_j \lambda'_{j,b} \mf F'_{j,y} \in R (A) .
\end{equation}
Bringing some indices to the other side in \eqref{eq:2.19}, we can arrange that all
the $\lambda_{i,b}$ and all the $\lambda'_{j,b}$ lie in $\Z_{\geq 0}$. Then \eqref{eq:2.41}
can be rewritten as
\[
HH_n \big( \bigoplus\nolimits_i \mf F_i^{\oplus \lambda_{i,b}} \big) = 
HH_n \big( \bigoplus\nolimits_j {\mc F'}_j^{\oplus \lambda_{j,b}} \big) .
\]
This is an instance of Lemma \ref{lem:2.2}.b.
\end{proof}

One can interpret Lemma \ref{lem:2.12} as: every algebraic family over $Y$ in $\C \otimes_\Z R (A)$
gives rise to a well-defined map on Hochschild homology.\\

We would like to realize the isomorphism from Theorem \ref{thm:2.1} with families of 
representations. To define the desired families of representations, we specialize to a setup
similar to root data. From now on $V$ will be finite dimensional complex vector space on
which $G$ acts linearly. We assume that we are given a family of ``parabolic" subgroups $G_P$ 
of $G$, indexed by the subsets of some finite set $\Delta$, such that
\begin{itemize}
\item $G_\emptyset = \{1\}$ and $G_\Delta = G$,
\item for every $P \subset \Delta$ there are $G_P$-stable algebraic subgroups $V_P$ and 
$V^P \subset V^{G_P}$, such that $V = V_P \oplus V^P$,
\item for $P \supset Q$: $G_P \supset G_Q, V_P \supset V_Q$ and $V^P \subset V^Q$
\item for any $P \subset \Delta$,
\begin{equation}\label{eq:2.13}
\Q \otimes_\Z \sum\nolimits_{Q \subsetneq P} \mr{ind}_{\mc O (V_P) \rtimes 
\C [G_Q,\natural]}^{\mc O (V_P) \rtimes \C [G_P,\natural]} R (\mc O (V_P) \rtimes \C [G_Q,\natural]) 
\end{equation}
has finite codimension in $\Q \otimes_\Z R (\mc O (V_P) \rtimes \C [G_P,\natural])$.
\end{itemize}
The second bullet entails that $\mc O (V) \cong \mc O (V^P) \otimes \mc O (V_P)$. For a 
representation $\delta$ of $\mc O (V_P) \rtimes \C [G_P,\natural]$ and $v \in V^P$, we define 
a representation $\C_v \otimes \delta$ of 
\[
\mc O (V^P) \otimes \mc O (V_P) \rtimes \C [G_P,\natural] = \mc O (V) \rtimes \C [G_P,\natural]
\quad \text{ by}
\]
\[
f_1 \otimes f_2 \otimes T_g \mapsto f_1 (v) \delta (f_2 \otimes T_g) \qquad
f_1 \in \mc O (V^P), f_2 \in \mc O (V_P), g \in G_P .
\]

\begin{defn}\label{def:2}
We call a finite dimensional representation $\delta$ of $\mc O (V_P) \rtimes \C [G_P,\natural]$ 
elliptic if it admits a $\mc O (V_P)^{G_P}$-character and does not belong to \eqref{eq:2.13}. 
For such $(P,\delta)$, the family of representations
\[
\pi (P,\delta,v) := \mr{ind}_{\mc O (V) \rtimes \C [G_P,\natural]}^{\mc O (V) \rtimes 
\C [G,\natural]} (\C_v \otimes \delta)  \qquad v \in V^P
\]
is called the algebraic family $\mf F (P,\delta)$. Its dimension is the dimension of $V^P$.
\end{defn}

The ellipticity condition implies that an algebraic family of this kind can not be extended
to a larger parameter space $V^{P'}$.
The construction of irreducible $\mc O (V_P) \rtimes \C [G_P,\natural]$-representations in 
\eqref{eq:2.24} shows that the $\mc O (V_P)^{G_P}$-character of an elliptic $\delta$ is just 
$0 \in V_P / G_P$. By the assumption on our parabolic subalgebras, there exist only finitely many 
such algebraic families with $\delta$ irreducible. In the remainder of this paragraph we abbreviate
\[
A = \mc O (V) \rtimes \C [G,\natural] .
\]

\begin{lem}\label{lem:2.5}
Let $\{ \mf F (P_i,\delta_i) \}_{i=1}^{n_{\mf F}}$ be a set of algebraic families of 
$A$-representations, such that
\[
\{ \pi (P_i,\delta_i,v_i) : v_i \in V^{P_i} , i = 1,\ldots, n_{\mf F} \}
\]
spans $\C \otimes_\Z R (A)$. Let $R^d (A) \subset R(A)$ be the $\Z$-span of the members
of the families $\mf F (P_i,\delta_i)$ of dimension $\geq d$. Then 
\[
\C \otimes_\Z R^d (A) \subset HH_0 (A)^*
\]
has a $\C$-basis
\[
\{ \nu_{g,v} : g \in \langle G \rangle, \dim V^g \geq d, v \in V^g / Z_G (g),
\natural^g |_{G_v \cap Z_G (g)} = 1 \} .
\]
\end{lem}
\begin{proof}
By \eqref{eq:2.12} every $\nu_{g,v}$ belongs to $\C \otimes_\Z R^0 (A)$. Consider a family
$\mf F (P_i,\delta_i)$ and an element 
\[
\omega \in ( \mc O (V^g) \otimes \natural^g )^{Z_G (g)} \subset HH_0 (A) ,
\]
such that $\dim V^{P_i} > \dim V^g$. Then 
\begin{equation}\label{eq:2.15}
v_i \mapsto \mr{tr}(\omega, \pi (P_i,\delta_i,v_i))
\end{equation}
is a regular function on $V^{P_i}$, like for any element of $HH_0 (A)$. But $\omega$ is not
defined outside $V^g$, so \eqref{eq:2.15} can only be zero there. Since $V^{P_i} \setminus V^g$
is dense in $V^{P_i}$, \eqref{eq:2.15} is zero everywhere. 

Consequently every $\nu_{g,v}$ lies in $\C \otimes_\Z R^{\dim V^g}(A)$, and the trace functions
from $R^d (A)$ are determined by their restrictions to
\[
\bigoplus\nolimits_{g \in \langle G \rangle : \dim V^g \geq d} 
(\mc O (V^g) \otimes \natural^g )^{Z_G (g)} .
\]
From this we see that $\C \otimes_\Z R^d (A)$ is exactly the $\C$-span of the $\nu_{g,v}$
with $\dim V^g \geq d$. By Lemma \ref{lem:2.4} these $\nu_{g,v}$ are linearly independent.
\end{proof}

Next we describe an algorithm to choose a minimal set of algebraic families of $A$-representations.
We start with the family $\mf F (\emptyset,\mr{triv})$ and proceed recursively. Suppose that for 
every dimension $D > d$ we have chosen a set of $D$-dimensional algebraic families 
$\mf F (P_i,\delta_i)$, where $i$ runs through some index set $I_D$, with the following property:
for generic $v_i \in V^{P_i}$ the set 
\[
\big\{ \pi (P_j ,\delta_j ,v_j) : j \in I_D, D > d,  cc(\pi (P_j ,\delta_j ,v_j)) = 
cc(\pi (P_i ,\delta_i ,v_i)) \big\}
\]
is linearly independent in $\Q \otimes_\Z R(G)^{\mf s}$. Here $cc$ denotes the 
$\mc O (V)^G$-character of an $A$-representation (if it exists). Moreover we regard 
$\mf F (P_j,\delta_j)$ here as a family in $R (A)$, which means that $\pi (P_j ,\delta_j ,v_j)$ 
and $\pi (P_j ,\delta_j ,v'_j)$ are considered as the same element if they have the same trace. 
This point of view is necessary for the linear independence criterion.  

Next we consider the set of $d$-dimensional algebraic families $\mf F (P'_i,\delta'_i)$. 
Suppose that for generic $v'_i \in V^{P'_i}$, the representation
$\pi (P'_i,\delta'_i,v'_i)$ is $\Q$-linearly independent from
\[
\big\{ \pi (P_j, \delta_j, v_j) : j \in I_D, D > d, cc(\pi (P_j, \delta_j, v_j)) =
cc(\pi (P'_i, \delta'_i, v'_i)) \big\} ,
\]
were we still regard $\mf F (P_j,\delta_j)$ as a family in $R (A)$.
Then we add $\mf F (P'_i,\delta'_i)$ to our collection of algebraic families.

Consider the remaining $d$-dimensional algebraic families. For $\mf F (P'_j, \delta'_j)$ we look at 
the same condition as for $\mf F (P'_i,\delta'_i)$, but now with respect to the index set 
$\cup_{D > d} I_D \cup \{i'\}$ instead of $\cup_{D > d} I_D$. If that condition is fulfilled, we add 
$\mf F (P'_j, \delta'_j)$ to our set of algebraic families. We continue this process until none of the
remaining $d$-dimensional algebraic families is (over generic points of that family) $\Q$-linearly
independent from the algebraic families that we chose already. At that point our set of 
$d$-dimensional algebraic families is complete, and we move on to families of dimension $d-1$.

In the end, this algorithm yields a collection 
\begin{equation}\label{eq:2.38}
\{ \mf F (P_i,\delta_i) : i \in I_d, 0 \leq d \leq \dim V \}
\end{equation}
such that:
\begin{itemize}
\item the representations $\{ \pi (P_i,\delta_i,v_i), i \in \cup_d I_d, v_i \in V^{P_i} \}$
span $\Q \otimes_\Z R(A)$,
\item if we remove any index from $\cup_d I_d$, the previous bullet becomes false,
\item for generic $v_i \in V^{P_i}$, $\pi (P_i,\delta_i,v_i)$ does not belong
to the span in $\Q \otimes_\Z R(A)$ of the other families $\mf F (P_j, \delta_j)$.
\end{itemize}

Notice that each $V^g$ is a vector space, and in particular an irreducible variety. That entails
that the $g \in G$ underlie a dichotomy, based the behaviour of the group
\[
Z_G (g,V^g) = \{ h \in Z_G (g) : h v = v \; \forall v \in V^g \} :
\]
\begin{itemize}
\item Suppose that $\natural^g (h) \neq 1$ for some $h \in Z_G (g,V^g)$. Then
\[
(\Omega^n (V^g) \otimes \natural^g )^{Z_G (g)} \subset
\Omega^n (V^g) \otimes (\natural^g )^{Z_G (g,V^g)} = 0
\] 
and the summand of $HH_n (A)$ indexed by $g$ is zero. 
\item Suppose that $Z_G (g,V^g) \subset \ker (\natural^g)$. Notice that $\dim (V^g)^k < 
\dim (V^g)$ for every $k \in Z_G (g) \setminus Z_G (g,V^g)$. Hence the set $\tilde{V^g}$ of 
$v \in V^g$ that are not fixed by any such $k$ is Zariski open and dense in $V^g$. The action 
of $Z_G (g)$ on $\tilde{V^g}$ factors through a free action of $Z_G (g) / Z_G (g,V^g)$, so it 
is easy to attain $Z_G (g)$-invariance on $\tilde{V^g}$. Therefore the summand of $HH_n (A)$ 
indexed by $g$ is nonzero.
\end{itemize}
To distinguish these cases, we say $g$ is $HH (A)$-irrelevant or $HH (A)$-relevant.
Recall from Lemma \ref{lem:2.5} that the set of trace functions
\[
\{ \nu_{g,v} : g \in \langle G \rangle, \dim V^g = d, v \in V^g / Z_G (g),
G_v \cap Z_G (g) \subset \ker (\natural^g) \}
\]
forms a $\C$-basis of $\C \otimes_\Z (R^d (A) / R^{d+1}(A) )$. For a fixed $HH(A)$-relevant $g$,
this gives an algebraic family of trace functions on $A$, supported on the sum of
the linear subspaces $T_{g'} \mc O (V)$ with $g'$ conjugate to $g$. Every member of this family
factors through $A / I_v^G$ for the appropriate maximal ideal $I_v^G$ of $\mc O (V)^G$, so
corresponds to a unique virtual $A$-representation with $\mc O (V)^G$-character $G v$. 

\begin{lem}\label{lem:2.7}
For $i = 1,\ldots,n_{\mf F}$ and $g \in \langle G \rangle$ there exist $\lambda_{g,i} \in \C$ and 
a map $\phi_{g,i} : V^g \to V^{P_i}$, given by some element of $G$, such that
\[
\nu_{g,v} = \sum\nolimits_{i : \dim (V^{P_i}) \geq \dim V^g} \lambda_{g,i} 
\mr{tr} \, \pi (P_i,\delta_i , \phi_{g,i} (v) )
\]
for all $v \in V^g$ with $G_v \cap Z_G (g) \subset \ker (\natural^g)$.
\end{lem}
\begin{proof}
To make the construction easier, we may omit some of the families $\mf F (P_i,\delta_i)$, so that 
the remaining families minimally span the part of $\C \otimes_\Z R(A)$ with $\mc O (V)^G$-characters 
in $G V^g / G$. Only families with $\dim (V^{P_i}) \geq \dim V^g$ can remain. 
This can be compared with the construction of the families of representations in \eqref{eq:2.38}.

Let us restrict to the generic part $\tilde{V^g}$ of $V^g$. There exist $\lambda_{g,i}(v) \in \C$ 
and $\phi_{g,i}(v)$ such that the lemma holds for $v \in \tilde{V^g}$. Since $\nu_{g,v}$ admits an
$\mc O (V)^G$ character and the remaining families are minimal in the above sense, every
$\phi_{g,i}(v)$ is unique up to applying some element of $G$ that stabilizes $\mf F (P_i,\delta_i)$.
We fix a generic $\tilde v \in \tilde{V^g}$ and we pick maps $\phi_{g,i}$ such that the required
property holds for $\nu_{g,\tilde v}$. By the uniqueness up to $G$ and the genericity, $\phi_{g,i}$
extends uniquely to continuous map $\tilde{V^g} \to V^{P_i}$.

Every $\phi_{g,i}$ preserves the $\mc O (V)^G$-characters, so is given by some element of $G$. 
In particular it is an injective regular map $V^g \to V^{P_i}$. As a representation of 
$\C [G,\natural]$, $\pi (P_i,\delta_i ,\phi_{g,j} (v))$ 
does not depend on $v$. The numbers $\lambda_{g,i}(v)$ are determined by the
earlier choices, so they do not depend on $v \in \tilde{V^g}$ either. 

Now the definition of the $\phi_{g,i}$ and the $\lambda_{g,i}$ applies to all $v \in V^g$. 
For all $i \in \{1,\ldots,n_{\mf F}\}$ that do not appear in this construction, we set
$\lambda_{g,i} = 0$.
All items in the statement of the lemma depend algebraically on $v$, so the validity of the
lemma extends from $\tilde{V^g}$ to all $v \in V^g$ for which $\nu_{g,v}$ is defined.
\end{proof}

\subsection{Hochschild homology via families of representations} \
\label{par:realiz}

Any algebraic family of representations $\mf F (P,\delta)$ gives rise to a 
homomorphism of $\mc O (V)^G$-algebras
\[
\begin{array}{cccc}
\mc F_{P,\delta} : & A & \to & \mc O (V^{P}) \otimes \mr{End}_\C ( W_{P,\delta}) \\
& f & \mapsto & [v \mapsto \pi (P,\delta,v)(f)] 
\end{array},
\]
where $W_{P,\delta}$ is the vector space underlying $\pi (P,\delta,v)$ for any $v \in V^{P}$. 
Consider a finite set of algebraic families $\mf F (P_i,\delta_i)$ whose members span
$\Q \otimes_\Z R (A)$, like in Lemma \ref{lem:2.5}.
All the $\mc F_{P_i,\delta_i}$ together induce a homomorphism of $\mc O (V)^G$-modules
\[
HH_n (\mc F_A) = HH_n \big( \bigoplus\nolimits_{i=1}^{n_{\mf F}} \mc F_{P_i,\delta_i} \big) : 
HH_n (A) \to \bigoplus\nolimits_{i=1}^{n_{\mf F}} \Omega^n (V^{P_i}) .
\]
We want to show that this map is injective and to describe its image. To that end, we set 
things up so that we can write down an inverse map.
For each $(g,i)$ as in Lemma \ref{lem:2.7}, $\phi_{g,i}$ yields an algebra homomorphism
\begin{equation}\label{eq:2.18}
\begin{array}{cccc}
\phi_{g,i}^* : & \mc O (V^{P_i}) \otimes \mr{End}_\C (W_{P_i,\delta_i}) & \to & 
\mc O (V^g) \otimes \mr{End}_\C (W_{P_i,\delta_i}) \\
& f \otimes A & \mapsto & f \circ \phi_{g,i} \otimes A 
\end{array}
\end{equation}
and an induced map on Hochschild homology
\[
HH_n (\phi_{g,i}^* ) : \Omega^n (V^{P_i}) \to \Omega^n (V^g). 
\]
Summing over all $g,i$ we obtain a homomorphism of $\mc O (V)^G$-modules
\begin{equation}\label{eq:2.30}
HH_n (\phi^*) = \bigoplus_{g \in \langle G \rangle} \sum_{i = 1}^{n_{\mf F}}
\lambda_{g,i} HH_n (\phi_{g,i}^*) : \bigoplus_{i=1}^{n_{\mf F}} \Omega^n (V^{P_i})
\to \bigoplus_{g \in \langle G \rangle} \Omega^n (V^g),
\end{equation}
where $\lambda_{g,i} = 0$ if $g$ is $HH(A)$-irrelevant or $\dim (V^g) > \dim (V^{P_i})$.

\begin{prop}\label{prop:2.8}
Fix a $HH(A)$-relevant $g \in \langle G \rangle$. The map
\[
\sum\nolimits_{i = 1}^{n_{\mf F}} \lambda_{g,i} HH_n (\phi_{g,i}^*) \circ HH_n (\mc F_A) \; : \;
HH_n (A) \longrightarrow \Omega^n (V^g)
\]
\enuma{
\item annihilates the summands $(\Omega^n (V^{g'}) \otimes \natural^{g'} )^{Z_G (g')}$
of $HH_n (A)$ with $g' \in \langle G \rangle \setminus \{g\}$,
\item is the identity on the summand $(\Omega^n (V^{g}) \otimes \natural^{g} )^{Z_G (g)}$
of $HH_n (A)$.
}
\end{prop}
\begin{proof}
We can express the map as
\[
\sum\nolimits_{i = 1}^{n_{\mf F}} \lambda_{g,i} HH_n (\phi_{g,i}^* \circ \mc F_{P_i,\delta_i}),
\]
where $\phi_{g,i}^* \circ \mc F_{P_i,\delta_i}$ is an algebraic family of representations.
By Lemma \ref{lem:2.7} the members of these families satisfy
\begin{equation}\label{eq:2.40}
\nu_{g,v} = \sum\nolimits_{i = 1}^{n_{\mf F}} \lambda_{g,i} 
\mr{tr} \, \pi (P_i,\delta_i, \phi (g,i)(v)) 
\end{equation}
whenever $\nu_{g,v}$ is defined. For $g' \in \langle G \rangle$ we consider the commutative 
algebra 
\[
A_{g'} := \C [\{T_{g^n} : n \in \Z\}] \otimes \mc O (V / (g'-1) V)
\]
As each $\pi (P_i,\delta_i,v_i)$ is obtained by induction from an $\mc O (V_{P_i}) \rtimes
\C [G_{P_i},\natural]$-representation on which $\mc O (V_{P_i})$ acts via evaluation at 0,
it is a semisimple $A$-representation and $\mc O (V)$ acts on $T_w V_{\delta_i}$ via 
evaluation at $w^{-1} v_i$. It follows that $\phi_{g,i}^* \circ \mc F_{P_i,\delta_i}$
consists of semisimple representations and can be decomposed as a direct sum of families
of $\mc O (V / (g'-1)V)$-representations of the form $f \mapsto f \circ w^{-1}$ for some
$w \in G$. Not all $w \in G$ appear here and some $w'$ may give the same family. We record
this as an ``isotypic" decomposition of $\mc O (V / (g'-1)V)$-representations 
\begin{equation}\label{eq:2.25}
\phi_{g,i}^* \circ \mc F_{P_i,\delta_i} = 
\bigoplus\nolimits_{w / \sim} (\phi_{g,i}^* \circ \mc F_{P_i,\delta_i} )_w 
\end{equation}
As $T_{g'}$ commutes with $\mc O (V / (g'-1)V)$, it stabilizes this decomposition.\\
(a) For $g' \neq g$ and $f \in \mc O (V / (g'-1)V)$ we have $\nu_{g,v}(T_{g'} f) = 0$.
In terms of \eqref{eq:2.25} that becomes
\begin{equation}\label{eq:2.26}
\sum\nolimits_{i = 1}^{n_{\mf F}} \lambda_{g,i} \sum\nolimits_{w / \sim} 
\mr{tr} \big( f(w^{-1}v) T_{g'}, (\pi (P_i,\delta_i, \phi_{g,i}(v) )_w \big) = 0.
\end{equation}
The subalgebra $\C [\{T_{g^n} : n \in \Z\}]$ has finite dimension and is semisimple, so
the restriction of the above representations to this subalgebra do not depend on $v \in V^{g'}$.
For $v \in V^g$ in generic position, we can separate the various $w/\!\!\sim$ in \eqref{eq:2.26},
which leads to 
\begin{equation}\label{eq:2.27}
\sum\nolimits_{i = 1}^{n_{\mf F}} \lambda_{g,i}
\mr{tr} \big( T_{g'}, (\pi (P_i,\delta_i, \phi_{g,i}(v) )_w \big) = 0
\end{equation}
for all $w / \!\! \sim$. By continuity that extends from generic $v$ to all $v \in V^{g'}$.
From \eqref{eq:2.27} we see that
\[
\sum\nolimits_{i = 1}^{n_{\mf F}} \lambda_{g,i} \mr{gtr} \circ
C_* (\phi_{g,i}^* \circ \mc F_{P_i,\delta_i}) \quad \text{annihilates} \quad
T_{g'} C_* \big( \mc O (V / (g'-1)) \big) .
\]
Combine that with \eqref{eq:2.23}.\\
(b) For $g' = g$, \eqref{eq:2.26} becomes
\begin{equation}\label{eq:2.28}
\sum\nolimits_{i = 1}^{n_{\mf F}} \lambda_{g,i} \sum\nolimits_{w / \sim} \mr{tr} 
\big( f(w^{-1}v) T_g, (\pi (P_i,\delta_i, \phi_{g,i}(v) )_w \big) = \nu_{g,v}(T_g f) .
\end{equation}
From \eqref{eq:2.23} and the $HH(A)$-relevance of $g$ we see that
\begin{equation}\label{eq:2.29}
\nu_{g,v}(T_g f) = [Z_G (g) : Z_G (g,V^g)]^{-1} 
\sum\nolimits_{h \in Z_G (g) / Z_G (g,V^g)} \natural^g (h) f (h^{-1} v).
\end{equation}
Comparing \eqref{eq:2.28} and \eqref{eq:2.29}, we deduce that every $w / \!\! \sim$ can be rewritten 
as a unique $h \in Z_G (g) / Z_G (g,V^g)$. Then \eqref{eq:2.28} becomes
\[
\sum\nolimits_{i = 1}^{n_{\mf F}} \lambda_{g,i} \sum\nolimits_{h \in Z_G (g) / Z_G (g,V^g)} \mr{tr} 
\big( f(h^{-1}v) T_g, (\pi (P_i,\delta_i, \phi_{g,i}(v) )_h \big) = \nu_{g,v}(T_g f) .
\]
Like in part (a) we can separate the various $h$, leading to
\[
\sum\nolimits_{i = 1}^{n_{\mf F}} \lambda_{g,i} \big( T_g, (\pi (P_i,\delta_i, \phi_{g,i}(v) )_h 
\big) = [Z_G (g) : Z_G (g,V^g)]^{-1} \natural^g (h) .
\]
Initially this holds only for generic $v$, but by continuity it extends to all $v \in V^g$. We deduce
\begin{equation*}
\sum_{i = 1}^{n_{\mf F}} \lambda_{g,i} \mr{gtr} \circ C_n (\phi_{g,i}^* \circ 
\mc F_{P_i,\delta_i}) (T_g \omega) = [Z_G (g) : Z_G (g,V^g)]^{-1} \hspace{-3mm} 
\sum_{h \in Z_G (g) / Z_G (g,V^g)} \hspace{-3mm} \natural^g (h) h \cdot \omega
\end{equation*}
for all $\omega \in C_n (\mc O (V / (g-1)V))$. In view of \eqref{eq:2.23}, this says exactly that
the map of the lemma is the identity on $T_g C_* (\mc O (V / (g-1)V))$.
\end{proof}

From \eqref{eq:2.38} and Lemmas \ref{lem:2.2} and \ref{lem:2.12} we see that 
\begin{equation}\label{eq:2.39}
HH_n (\phi^*) \circ HH_n (\mc F_A) \; : \; HH_n (A) \to 
\bigoplus\nolimits_{g \in \langle G \rangle} \Omega^n (V^g)
\end{equation}
can be considered as evaluation at the families of virtual $A$-representations $\nu_{g,v}$
(extended naturally to all $v \in V^g$).

From now on we assume that our collection of algebraic families $\mf F (P_i,\delta_i)$ has been
chosen in a minimal way, as in \eqref{eq:2.38}.

\begin{lem}\label{lem:2.9}
Under the above assumption, the map
\[
HH_n (\phi^*) \; : \; \bigoplus\nolimits_{i=1}^{n_{\mf F}} \Omega^n (V^{P_i}) \to 
\bigoplus\nolimits_{g \in \langle G \rangle} \Omega^n (V^g)
\]
is injective.
\end{lem}
\begin{proof}
Consider a nonzero $\sum_{i=1}^{n_{\mf F}} \omega_i \in \bigoplus_{i=1}^{n_{\mf F}} \Omega^n (V^{P_i})$.
We select an index $j$ and a small open set $U$ (for the analytic topology) in $V^{P_j}$ such that 
$\omega_j (u) \neq 0$ for all $u \in U$. Since the set of generic points in $V^{P_j}$ is open and 
dense, we may assume that $\{ \pi (P_j,\delta_j,u) : u \in U\}$ does not share any
$\mc O (V)^G$-characters with any family $\mf F (P_i,\delta_i)$ of lower dimension, and that
$w U \cap U$ is empty unless $w \in Z_G (g,V^g)$. 

The construction of a minimal set of algebraic families of $A$-representations entails that
none of the representations $\{ \pi (P_j,\delta_j,u) : u \in U\}$ belongs to the span in
$\Q \otimes_\Z R(A)$ of the other families $\mf F (P_i,\delta_i)$. The same holds for any
linear combination of these representations, because $U$ does not contain two points from any
$Z_G (g)$-orbit.

Now Lemma \ref{lem:2.4} shows there must exist a $g \in \langle G \rangle$ with 
$\lambda_{g,j} \neq 0$. The component of $HH_n (\phi^*)$ indexed by $g$ is
$\sum_{i=1}^{n_{\mf F}} \lambda_{g,i} HH_n (\phi_{g,i}^*)$, so
\[
(HH_n (\phi^*) \omega ) \big|_{\phi_{g,j} (U)} = \sum\nolimits_{i=1}^{n_{\mf F}} \lambda_{g,i}
HH_n (\phi_{g,i}^*) \omega |_{\phi_{g,i}^{-1} \phi_{g,j} U} .
\] 
This is nonzero by the above linear independence property of the set
$\{ \pi (P_j,\delta_j,u) : u \in U\}$.
\end{proof}

With Proposition \ref{prop:2.8} and Lemma \ref{lem:2.9} we can provide a description of $HH_n (A)$ 
in the style of \cite{BDK}. Recall that $V$ is a finite dimensional $G$-representation, that 
$A = \mc O (V) \rtimes \C [G,\natural])$ and that $\mf F (P_i,\delta_i) \; (i=1,\ldots,n_{\mf F})$ 
are algebraic families of $A$-representations whose members span $\C \otimes_\Z R(A)$ in a 
minimal way.

\begin{thm}\label{thm:2.6}
\enuma{
\item The homomorphism of $\mc O (V)^G$-modules
\[
HH_n (\mc F_A) = \bigoplus\nolimits_{i=1}^{n_{\mf F}} HH_n (\mc F_{P_i,\delta_i}) \;:\;
HH_n (A) \to \bigoplus\nolimits_{i=1}^{n_{\mf F}} \Omega^n (V^{P_i})
\] 
is injective. The homomorphism of $\mc O (V)^G$-modules
\[
HH_n (\phi^*) \;:\; HH_n (\mc F_A) HH_n (A) \to \bigoplus\nolimits_{g \in \langle G \rangle} 
\big( \Omega^n (V^{g}) \otimes \natural^{g} \big)^{Z_G (g)}
\]
is bijective.
\item In degree $n=0$, the condition on $\omega = \sum_{i=1}^{n_{\mf F}} \omega_i \in 
\bigoplus_{i=1}^{n_{\mf F}} \mc O (V^{P_i})$ that 
describes the image is: whenever $\lambda_j \in \C, i_j \in \{1,\ldots,n_{\mf F}\}, v_{i_j} \in 
V^{P_{i_j}}$ and $\sum_j \lambda_j \pi (P_{i_j},\delta_{i_j},v_{i_j}) = 0$ in $\C \otimes_\Z R(A)$, 
also $\sum_j \lambda_j \omega_{i_j} (v_{i_j}) = 0$. This determines an isomorphism of $Z(A)$-modules
\[
HH_0 (A) \cong (\C \otimes_\Z R (A))^*_{\mr{reg}} .
\]
}
\end{thm}
\begin{proof}
(a) Proposition \ref{prop:2.8} entails that $HH_n (\phi^*) \circ HH_n (\mc F_A)$ is the identity on
\[
\bigoplus\nolimits_{g \in \langle G \rangle} (\Omega^n (V^{g}) \otimes \natural^{g} )^{Z_G (g)}.
\]
Hence $HH_n (\mc F_A)$ is injective, and from Lemma \ref{lem:2.9} we obtain the desired bijectivity 
of $HH_n (\phi^*)$.\\
(b) The image of $HH_0 (\phi^*) \circ HH_0 (\mc F_A)$ is 
\begin{equation}\label{eq:2.17}
\bigoplus\nolimits_{g \in \langle G \rangle} (\mc O (V^{g}) \otimes \natural^{g} )^{Z_G (g)}.
\end{equation}
The map associated to $\omega \in \bigoplus_{i=1}^{n_{\mf F}} \mc O (V^{P_i})$ in the statement sends 
\begin{equation}\label{eq:2.37}
\nu_{g,v} \quad \text{to} \quad 
\sum\nolimits_{i=1}^{n_{\mc F}} \omega_i (\phi_{g,i}(v)) = (HH_0 (\phi^*) \omega )(g,v) .
\end{equation}
The $\nu_{g,v}$ satisfy the relations \eqref{eq:2.16}, so $\omega$ must respect these in order to
descend to a function on $\C \otimes_\Z R(A)$. In view of \eqref{eq:2.37}, that means that
$HH_0 (\phi^*) \omega$ must be $Z_G (g)$-invariant. By construction the image of $HH_0 (\phi^*)$
consists of regular functions, so $\omega$ must belong to \eqref{eq:2.17}.

On the other hand, from Lemma \ref{lem:2.3} we know that the maximal ideal spectrum of \eqref{eq:2.17} 
is in bijection with the set of all $\nu_{g,v}$, modulo the relations \eqref{eq:2.16}. By
\eqref{eq:2.12} the resulting quotient set forms a basis of $\C \otimes_\Z R(A)$.
Hence \eqref{eq:2.17} can be considered as a subset of the linear dual space 
$(\C \otimes_\Z R(A))^*$ and $HH_0 (\phi^*)^{-1}$ of \eqref{eq:2.17} is the set of all
elements that satisfy the conditions stated in the theorem.

For $a \in A$ and $\omega = HH_0 (\mc F_A)(a)$, the definition of the generalized trace map gtr
shows that
\begin{equation}\label{eq:2.45}
\omega (\pi (P_i,\delta_i,v_i)) = \mr{tr} \, (\pi (P_i,\delta_i,v_i) a) .
\end{equation}
Hence the map
\[
HH_0 (A) \to HH_0 (\mc F_A) HH_0 (A) \to (\C \otimes_\Z R (A))^*
\]
constructed above is just the $Z(A)$-linear map \eqref{eq:2.43}. It is injective because 
$HH_0 (\mc F_A)$ is injective and the values \eqref{eq:2.45} can be recovered from the image
of $\omega$ in $(\C \otimes_\Z R (A))^*$. We know from \eqref{eq:2.44} that the image of
\eqref{eq:2.43} is contained in $(\C \otimes_\Z R (A))^*_{\mr{reg}}$. Conversely every
element $f \in (\C \otimes_\Z R (A))^*_{\mr{reg}}$ yields a regular function on $V^{P_i}$ via
pairing with $\mf F_{P_i,\delta_i}$, so $f$ comes from an element of 
$\bigoplus\nolimits_{i=1}^{n_{\mf F}} \Omega^n (V^{P_i})$.
\end{proof}

With the equality $\nu_{h g h^{-1},hv} = \natural^g (h) \nu_{g,v}$ from \eqref{eq:2.16} 
we can extend Lemma \ref{lem:2.7} from $g \in \langle G \rangle$ to all
$g \in G$. Namely, for $g \in \langle G \rangle$ and $h \in G$ we define
\begin{equation}\label{eq:2.35}
\lambda_{h g h^{-1},i} = \natural^g (h) \lambda_{g,i} \quad \text{and} \quad
\phi_{h g h^{-1},i} = \phi_{g,i} \circ h^{-1} .
\end{equation}
That yields a variation on \eqref{eq:2.30}:
\begin{equation}\label{eq:2.36}
HH_n (\tilde \phi^*) := \bigoplus_{g \in G} \sum_{i = 1}^{n_{\mf F}}
\lambda_{g,i} HH_n (\phi_{g,i}^*) \;:\; \bigoplus_{i=1}^{n_{\mf F}} \Omega^n (V^{P_i})
\to \bigoplus_{g \in G} \Omega^n (V^g),
\end{equation}
We note that this map hardly differs from $HH_n (\phi^*)$, because it is entirely determined
by the components indexed by $g \in \langle G \rangle$ via the actions from \eqref{eq:2.32}.
From Theorem \ref{thm:2.6} and \eqref{eq:2.33} we conclude:

\begin{cor}\label{cor:2.11}
There is a $\mc O (V)^G$-linear bijection
\[
HH_n (\tilde \phi^*) \circ HH_n (\mc F_A) \;:\; HH_n (A) \to 
\Big( \bigoplus\nolimits_{g \in G} \Omega^n (V^g) \otimes \natural^g \Big)^G .
\]
\end{cor} 
This realizes the isomorphism Theorem \ref{thm:2.1} in terms of algebra homomorphisms.

\begin{ex}\label{ex:2.B}
We continue Example \ref{ex:2.A}, so $G = Q_8 / Z(Q_8)$ acts on $V = \C^2$ by reflections. 
We note that $V^{\pm \mb k} = \{(0,0)\}, Z_G (\pm \mb{k},V^{\pm \mb k}) = G$ and 
$\natural^{\pm \mb k}$ is nontrivial. Hence $\pm \mb{k}$ is $HH (A)$-irrelevant, where
$A = \mc O (V) \rtimes \C [G,\natural]$.

As parabolic subgroups we take 
\[
G_\emptyset = \{ \pm \mb{1} \} ,\; G_{\mb i} = \{ \pm \mb{1}, \pm \mb{i} \} ,\;
G_{\mb j} = \{ \pm \mb{1}, \pm \mb{j} \} ,\; G_{\mb i,\mb j} = G .
\]
In each case $V^P = V^{G_P}$ and $V_P$ is the orthogonal complement to $V^P$. To span 
$\Q \otimes_\Z R(A)$ we need three algebraic families of representations, for instance:
\begin{itemize}
\item $\mf F (\emptyset ,\mr{triv} ) = \{ \mr{ind}_{\mc O (V)}^A \C_v : v \in V \}$.
\item Define $\delta_{\mb i} \in \Irr (\mc O (V_{\mb i}) \rtimes \C [G_{\mb i},\natural]$ by
$\delta_{\mb i}(T_{\pm \mb i}) = i$ and $\delta_i (f) = f(0)$ for $f \in \mc O (V_{\mb i})$.
Take
\[
\mf F (\mb i, \delta_{\mb i}) = \{ \mr{ind}_{\mc O (V) \rtimes \C [G_{\mb i},\natural]}^A 
(\C_v \otimes \delta_{\mb i}) : v \in V^{\mb i} = \{0\} \times \C \}
\]
\item Define $\delta_{\mb j} \in \Irr (\mc O (V_{\mb j}) \rtimes \C [G_{\mb j},\natural]$ by
$\delta_{\mb j}(T_{\pm \mb i}) = -i$ and $\delta_j (f) = f(0)$ for $f \in \mc O (V_{\mb j})$.
Take
\[
\mf F (\mb j, \delta_{\mb j}) = \{ \mr{ind}_{\mc O (V) \rtimes \C [G_{\mb j},\natural]}^A 
(\C_v \otimes \delta_{\mb j}) : v \in V^{\mb j} = \{0\} \times \C \}
\]
\end{itemize}
Composing $\C_v \otimes \delta_{\mb i}$ with conjugation by $T_{\pm \mb j}$ gives
\[
\C_{-v} \otimes \delta_{-\mb i} : f \otimes T_{\pm \mb i} \mapsto -i f(-v) .
\]
Hence $\mf F (\mb i,\delta_{\mb i})$ is not stable under elements of $G \setminus G_{\mb i}$,
and similarly for $\mf F (\mb j, \delta_{\mb j})$. The only relations in $\Q \otimes_\Z R(A)$
between the members of these families are:
\begin{align*}
& \pi (\emptyset, \mr{triv}, v_{\mb i}) = \pi (\mb i, \delta_{\mb i}, v_{\mb i}) + 
\pi (\mb i, \delta_{-\mb i}, v_{\mb i}) = \pi (\mb i, \delta_{\mb i}, v_{\mb i}) + 
\pi (\mb i, \delta_{\mb i}, -v_{\mb i}) \qquad v_{\mb i} \in V^{\mb i}, \\
& \pi (\emptyset, \mr{triv}, v_{\mb j}) = \pi (\mb j, \delta_{\mb j}, v_{\mb j}) + 
\pi (\mb j, \delta_{-\mb j}, v_{\mb j}) = \pi (\mb j, \delta_{\mb j}, v_{\mb j}) + 
\pi (\mb j, \delta_{\mb j}, -v_{\mb j}) \qquad v_{\mb j} \in V^{\mb j}, \\
& \pi (\emptyset, \mr{triv},0) = 2 \pi (\mb i, \delta_{\mb i}, 0 ) =
2 \pi (\mb j, \delta_{\mb j}, 0) .
\end{align*}
Comparing traces of representations we find
\begin{align*}
& 4 \nu_{\pm \mb{1},v} = \pi (\emptyset, \mr{triv},v)  \qquad v \in V ,\\
& 4 \nu_{\pm \mb{i},v_{\mb i}} = -2i \pi (\mb i, \delta_{\mb i},v_{\mb i}) +
i \pi (\emptyset, \mr{triv},v_{\mb i}) \qquad v_{\mb i} \in V^{\mb i},\\
& 4 \nu_{\pm \mb{j},v_{\mb j}} = -2i \pi (\mb j, \delta_{\mb j},v_{\mb j}) +
i \pi (\emptyset, \mr{triv},v_{\mb j}) \qquad v_{\mb j} \in V^{\mb j}.
\end{align*}
Notice that $\nu_{\pm \mb{i},0} = \nu_{\pm \mb{j},0} = 0 \neq \nu_{\pm \mb{1},0}$.
In case $V^g \subset V^{\mb j}$, $\phi_{g,\mb j}$ equals
\[
\mr{Res}^{V^{\mb j}}_{V^g} : \mc O (V^{\mb j}) \otimes
\End_\C (\C [G,\natural] \underset{\C[G_{\mb j},\natural]}{\otimes} \C) \to \mc O (V^g) \otimes
\End_\C (\C [G,\natural] \underset{\C[G_{\mb j},\natural]}{\otimes} \C).
\]
The maps $\phi_{g,\emptyset}$ and $\phi_{g,\mb i}$ admit similar descriptions (for the latter
provided that $V^g \subset V^{\mb i})$. Thus $HH_n (\phi^*)$ equals
\[
\sum\nolimits_{g \in G} \frac{1}{4} HH_n (\phi_{g,\emptyset}) + \frac{i}{4} 
HH_n (\phi_{\pm \mb{1},\mb i}) - \frac{i}{2} HH_n (\phi_{\pm \mb{i},\mb i}) - 
\frac{i}{4} HH_n (\phi_{\pm \mb{1},\mb j}) + \frac{i}{2} HH_n (\phi_{\pm \mb{j},\mb j}) .
\]
Theorem \ref{thm:2.6} provides an injection
\[
HH_n (\mc O (V) \rtimes \C [G,\natural]) \to \Omega^n (V) \oplus \Omega^n (V^{\mb i})
\oplus \Omega^n (V^{\mb j})
\]
whose image is precisely
\[
HH_n (\phi^* )^{-1} \Big( \Omega^n (V)^G \oplus ( \Omega^n (V^{\mb i}) \otimes \natural^{\pm \mb i}
)^G \oplus (\Omega^n (V^{\mb j}) \otimes \natural^{\pm \mb j} )^G \Big) .
\]
We note that in degree $n=1$ the specialization of $HH_1 (A)$ at $Gv = (0,0)$ is 
$0 \oplus \C \textup{d}z_2 \oplus \C \textup{d}z_1$. Remarkably, the dimension of this space is
larger than the number of irreducible $A$-representations with $\mc O (V)^G$-character $(0,0)$.\\
\end{ex}

Almost all results in Section \ref{sec:crossed} are also valid in a smooth 
setting. Let $V$ be a smooth manifold on which $G$ acts by diffeomorphisms, so that $G$ also acts on 
$C^\infty (G)$. We compute the Hochschild homology of $C^\infty (V) \rtimes \C [G,\natural]$,
with respect to the completed bornological tensor product or equivalently the completed projective 
tensor product. Recall that the smooth version of the HKR theorem was proved by Connes: 
\[
HH_n (C^\infty (V)) \cong \Omega_{sm}^n (V) ,
\]
where $\Omega_{sm}^n$ stands for smooth differential forms of degree $n$. The results of Paragraph
\ref{par:HHcrossed} hold in that setting, for Paragraphs \ref{par:famCrossed}--\ref{par:realiz}
our results remain valid in a smooth setting with $V$ a real vector space.

\section{Twisted graded Hecke algebras}
\label{sec:HHGHA}

We will adapt the computations from Paragraph \ref{par:HHcrossed} to graded Hecke algebras extended 
with a twisted group algebra. Let $(X,\Phi,Y,\Phi^\vee,\Delta)$ be a based root datum
with Weyl group $W = W(\Phi)$. We write
\[
\mf t_\R = \R \otimes_Z Y, \mf t = \C \otimes_\Z Y ,\qquad \mf t^* = \C \otimes_\Z X .
\]
Let $\Gamma$ be a finite group acting on the root datum and let $\natural : \Gamma \times
\Gamma \to \C^\times$ be a 2-cocycle. We regard it also as a 2-cocycle of the group
$W \Gamma = W \rtimes \Gamma$. Let $k : \Delta \to \C$ be a $W\Gamma$-invariant parameter
function. The twisted graded Hecke algebra $\mh H (\mf t, W\Gamma,k,\natural)$ associated
to these data is the vector space $\mc O (\mf t) \otimes \C [W\Gamma,\natural]$ with
multiplication defined by
\begin{itemize}
\item $\mc O (\mf t)$ and $\C [W\Gamma,\natural]$ are embedded as unital subalgebras,
\item for $\alpha \in \Delta$ and $f \in \mc O (\mf t)$:
\[
f T_{s_\alpha} - T_{s_\alpha} s_\alpha (f) = k_\alpha (f - s_\alpha (f)) \alpha^{-1} ,
\]
\item for $\gamma \in \Gamma$ and $f \in \mc O (\mf  t)$: 
\[
T_\gamma f T_\gamma^{-1} = \gamma (f) = 
[\lambda \mapsto f (\gamma^{-1} \lambda)] \qquad \lambda \in \mf t .
\]
\end{itemize}
When $\natural$ is trivial, we omit it from the notation and we speak of a graded
Hecke algebra (or an extended grade Hecke algebra if $\Gamma$ is nontrivial). This relates 
to the notation in the introduction by
\[
\mh H (\mf t, W\Gamma,k,\natural) = \mh H (\mf t,W,k) \rtimes \C[\Gamma,\natural] .
\]
Notice that for $k = 0$ we recover the twisted crossed product $\mc O (\mf t) \rtimes 
\C [W\Gamma,\natural]$. 
Multiplication with $\epsilon \in \C^\times$ defines a bijection
$m_\epsilon : \mf t^* \to \mf t^*$, which extends to an algebra automorphism of $\mc O (\mf t)$.
From the above multiplication rules we see that it extends even further, to an algebra
isomorphism
\begin{equation}\label{eq:4.1}
m_\epsilon : \mh H (\mf t, W\Gamma,\epsilon k,\natural) \to
\mh H (\mf t, W\Gamma,k,\natural) 
\end{equation}
which is the identity on $\C [W\Gamma,\natural]$. For $\epsilon = 0$ the homomorphism
$m_0$ is well-defined, but not bijective.
Like in \eqref{eq:2.20}, let 
\[
1 \to Z \to \tilde \Gamma \to \Gamma \to 1
\]
be a finite central extension such that $\natural$ becomes trivial in $H^2 (\tilde \Gamma ,\C^\times)$,
and let $p_\natural \in \C [Z]$ be the associated minimal central idempotent. Then
\[
\mh H (\mf t, W\Gamma,k,\natural) \cong p_\natural \mh H (\mf t, W \tilde \Gamma,k) ,
\]
a direct summand of the extended graded Hecke algebra
\begin{equation}\label{eq:4.2}
\mh H (\mf t, W \tilde \Gamma,k) = \mh H (\mf t, W,k ) \rtimes \tilde \Gamma .
\end{equation}
The Hochschild homology of \eqref{eq:4.2} was computed in \cite[Theorem 3.4]{SolHomGHA}. 
It is isomorphic to $HH_* (\mc O (\mf t) \rtimes W \tilde \Gamma)$, which we already know 
from \eqref{eq:2.2}. The arguments in \cite{SolHomGHA} make use of the subcomplexes
\begin{equation}\label{eq:4.3}
w C_* (\mc O (\mf t / (w-1) \mf t)) \quad \text{of} \quad C_* (\mh H (\mf t,W \tilde \Gamma,k)) .
\end{equation}
For each $w \in \langle W \tilde \Gamma \rangle$, this subcomplex contributes precisely
$\Omega^n (\mf t^w )^{Z_{W \tilde \Gamma} (w)}$ to\\ $HH_n (\mh H (\mf t,W \tilde \Gamma,k))$,
and $HH_n (\mh H (\mf t,W \tilde \Gamma,k))$ is the direct sum of these contributions.
This works for every parameter function $k$, and in particular yields a canonical $\C$-linear
bijection
\begin{equation}\label{eq:4.4}
HH_n (\mh H (\mf t,W \tilde \Gamma,k)) \to HH_n (\mc O (\mf t) \rtimes W \tilde \Gamma) .
\end{equation}
The constructions involved in \eqref{eq:4.4} affect neither $\C [Z]$ nor the central
idempotent $p_\natural$. Hence 
\begin{multline}\label{eq:4.5}
HH_n ( \mh H (\mf t, W\Gamma,k,\natural)) \cong HH_n ( p_\natural \mh H (\mf t, W \tilde \Gamma,k))
= p_\natural HH_n ( \mh H (\mf t, W \tilde \Gamma,k)) \cong \\
p_\natural HH_n (\mc O (\mf t) \rtimes W \tilde \Gamma) = HH_n (p_\natural \mc O (\mf t) 
\rtimes W \tilde \Gamma) \cong HH_n (\mc O (\mf t) \rtimes \C [W \tilde \Gamma,\natural]) .
\end{multline}
The second line of \eqref{eq:4.5} is an instance of Theorem \ref{thm:2.1}. We conclude that there
is an isomorphism of vector spaces
\begin{equation}\label{eq:4.6}
HH_n ( \mh H (\mf t, W\Gamma,k,\natural)) \cong \bigoplus\nolimits_{w \in \langle W \Gamma \rangle}
\big( \Omega^n (\mf t^w) \otimes \natural^w \big)^{Z_{W \Gamma}(w)} , 
\end{equation}
where the summand indexed by $w$ arises from the differential complex\\
$T_w C_* (\mf t / (w-1) \mf t)$, which does not depend on $k$.

Although the isomorphism \eqref{eq:4.6} is in general not  canonical, we see from the proofs 
of \cite[Theorem 3.4]{SolHomGHA} and Theorem \ref{thm:2.1} that it depends only on some choices in 
$\C [W\tilde \Gamma ,\natural]$. These choices can be made independently of $k$, so \eqref{eq:4.6} 
provides an identification with $HH_n ( \mh H (\mf t, W\Gamma,k',\natural))$ for any parameter
function $k' : \Delta \to \C$.  Unfortunately this isomorphism does not come from an algebra 
homomorphism, which makes it difficult to handle.

We warn that usually \eqref{eq:4.6} is not an isomorphism of $\mc O (\mf t)^{W \Gamma}$-modules, 
for the $\mc O (\mf t)^{W \Gamma}$-module structure on $HH_n ( \mh H (\mf t, W\Gamma,k,\natural))$ 
is a bit more complicated than suggested by \eqref{eq:4.6}.

\subsection{Representation theory} \ 

Like in Theorem \ref{thm:2.6} we want to obtain an expression for $HH_n ( \mh H (\mf t, W\Gamma,
k,\natural))$ in terms of algebraic families of representations. For $\mh H (\mf t, W,\natural) 
\rtimes  \tilde \Gamma$ that was achieved in \cite[Theorem 3.1]{SolHomAHA}.
However, the families of (virtual) representations used in \cite{SolHomAHA} do not seem to
be available in our more general setting with a nontrivial 2-cocycle. To overcome that we will
modify some arguments from \cite{SolHomAHA}, so that they become available in larger generality.

Firstly, we need to specify our parabolic subalgebras. For every $P \subset \Delta$ there is a
standard parabolic subalgebra $\mh H (\mf t,W_P,k)$ of $\mh H (\mf t,W,k)$. But that does not
yet mimic the situation for reductive groups well enough. To that end we need to allow
several parabolic subalgebras with underlying root datum $(X,\Phi_P,Y,\Phi_P^\vee,P)$, namely
of the form $\mh H (\mf t, W_P \rtimes \Gamma')$ where $\Gamma' \subset \Gamma$ stabilizes $P$.
More precisely, we assume that we are given a finite set $\Delta'$ with a surjection to
$\Delta$ (written as $Q \mapsto \Delta_Q$) and for each $Q \subset \Delta'$ a subgroup
$\Gamma_Q \subset \Gamma$ stabilizing $\Delta_Q$. We abbreviate the group
$W_{\Delta_Q} \rtimes \Gamma_Q$ to $(W \Gamma)_Q$. Furthermore, we assume that the collection
of parabolic subalgebras 
\[
\mh H^Q = \mh H (\mf t, (W \Gamma)_Q, k, \natural) \quad \text{of} \quad
\mh H = \mh H (\mf t, W \Gamma, k, \natural)
\]
satisfies the conditions on page \pageref{eq:2.13}. Here the role of $V$ is played by the
vector space $\mf t$ and
\[
\mf t^Q \subset \mf t^{(W \Gamma)_Q} ,\qquad \C \Delta_Q \subset \mf t_Q ,\qquad
\mf t^Q \oplus \mf t_Q = \mf t .
\]
Let us abbreviate $\mh H_Q = \mh H (\mf t_Q, (W \Gamma)_Q, k, \natural)$, so that
\[
\mh H^Q = \mc O (\mf t^Q) \otimes_\C \mh H_Q \quad \text{as algebras}.
\] 
With the following result we will reduce several issues for $\mh H$ to the simpler algebra
\[
\mh H (\mf t,W \Gamma,0,\natural) = \mc O (\mf t) \rtimes \C[W \Gamma, \natural].
\]

\begin{thm}\label{thm:4.1}
Assume that $k_\alpha \in \R$ for every $\alpha \in \Delta$, and let $\epsilon \in \R_{\geq 0}$.
There exists a natural bijection
\[
\zeta_\epsilon : \Q \otimes_\Z R (\mh H (\mf t, W\Gamma,k,\natural) ) \to 
\Q \otimes_\Z R (\mh H (\mf t, W \Gamma, \epsilon k, \natural)) ,
\]
and similarly for all its parabolic subalgebras, with the following properties:
\begin{enumerate}[(i)]
\item $\zeta_\epsilon (\pi)$ is a tempered virtual $\mh H (\mf t, W \Gamma, \epsilon k, \natural)
$-representation if and only if $\pi$ is a tempered virtual $\mh H$-representation;
\item $\zeta_\epsilon$ commutes with parabolic induction, in the sense that
\[
\zeta_\epsilon \big( \mr{ind}_{\mh H^Q}^{\mh H} (\C_\lambda \otimes \sigma) \big) =
\mr{ind}_{\mh H (\mf t, (W \Gamma)_Q, \epsilon k, \natural)}^{\mh H (\mf t, W \Gamma, 
\epsilon k, \natural)} \big( \C_\lambda \otimes \zeta_\epsilon (\sigma) \big)
\]
for a tempered $\sigma \in R(\mh H_Q)$ and $\lambda \in \mf t^Q$;
\item if $\lambda \in \sqrt{-1} \mf t_\R$ and $\pi$ is a virtual representation with
$\mc O (\mf t)^{W\Gamma}$-character in $W \Gamma \lambda + \mf t_\R$, then so is
$\zeta_\epsilon (\pi)$;
\item if $\pi$ is tempered and admits a $\mc O (\mf t)^{W\Gamma}$-character in 
$\mf t_\R$, then $\zeta_\epsilon (\pi) = \pi \circ m_\epsilon$, with $m_\epsilon$
as in \eqref{eq:4.1};
\item $\zeta_\epsilon$ preserves the underlying $\C [W\Gamma,\natural]$-representations.
\end{enumerate}
\end{thm}
\begin{proof}
In \cite[\S 2.3]{SolAHA} and \cite[Theorem 2.4]{SolHomAHA} this was proven for
$\mh H (\mf t, W\Gamma,k)$, with $\epsilon = 0$. Item (v) is not mentioned explicitly there, but 
it is a direct consequence of the properties (ii) and (iv). With that at hand, we can restrict
\[
\zeta_0 : \Q \otimes_\Z R (\mh H (\mf t, W\Gamma,k)) \to
\Q \otimes_\Z R (\mh H (\mf t, W\Gamma,0))
\]
to the image of $p_\natural$ on both sides. That yields the desired map $\zeta_0$.

Now we consider $\epsilon \in \R_{>0}$, which is actually easier, because the two involved
algebras are isomorphic via $m_\epsilon$. This case is not mentioned in \cite{SolAHA} or
\cite{SolHomAHA}, but it can be derived from related results for affine Hecke algebras
\cite[\S 4]{SolAHA} similarly to the case $\epsilon = 0$. Alternatively one can obtain
$\zeta_\epsilon$ as $(\zeta'_0)^{-1} \zeta_0$, where $\zeta'_0$ means $\zeta_0$ for the
algebra $\mh H (\mf t, W \Gamma, \epsilon k, \natural)$.
\end{proof}

To make full use of Theorem \ref{thm:4.1}, we assume from now on that $k_\alpha \in \R$ for all
$\alpha \in \Delta$. We note that the maps $\zeta_\epsilon$ in Theorem \ref{thm:4.1} are well-defined 
and bijective for any $\epsilon \in \C$. Only for $\epsilon \notin \R_{\geq 0}$ they have fewer 
nice properties with respect to temperedness, see \cite[\S 2.2]{SolHomAHA}.

Let $\Rep_{f,t}(\mh H)$ be the category of finite dimensional tempered $\mh H$-modules. For a 
discrete series representation $\delta$ (see \cite[\S 5]{SolHomGHA}) of a parabolic subalgebra 
$\mh H_Q$, let $\Rep_{f,t}^{[Q,\delta]}(\mh H)$ be the full subcategory of $\Rep_f (\mh H)$ 
generated by the subquotients of the representations 
\[
\pi (Q,\delta,\lambda) = \mr{ind}_{\mh H^Q}^{\mh H}(\C_\lambda \otimes \delta) \text{ with }
\lambda \in \sqrt{-1} \mf t^Q_\R .
\] 
Choose a set $\Delta_{\mh H}$ of representatives $\mf d = [Q,\delta]$ for such pairs up to 
$W\Gamma$-equivalence. Define
\[
(W \Gamma)_{\mf d} \text{ as the stabilizer of } (Q,\delta) .
\]

\begin{thm}\label{thm:4.7}
The Grothendieck group $R_t (\mh H)$ of $\Rep_{f,t}(\mh H)$ decomposes a direct sum 
$\bigoplus_{\mf d \in \Delta_{\mh H}} R_t (\mh H)^{\mf d}$.
\end{thm}
\begin{proof}
By the main result of \cite{DeOp}, an analogue of the Plancherel isomorphism \cite{Wal}
for affine Hecke algebras, our theorem holds for affine Hecke algebras with positive parameters. 
That extends to twisted affine Hecke algebras $\mc H$ with positive parameters
\cite[Theorem 3.2.2]{SolAHA}. Choose such an algebra $\mc H$, whose associated graded Hecke
algebra, via the localization process from \cite[\S 2.1]{SolAHA}, is $\mh H$. This is possible 
because $\mh H$ has real parameters. Then \cite[Theorem 2.1.3]{SolAHA} provides an equivalence
between the subcategory of $\Rep_f (\mh H)$ formed by the representations with all their
$\mc O (\mf t)$-weights in a certain analytically open neighborhood $U$ of $\mf t_\R$ in $\mf t$,
and an analogous category of $\mc H$-modules. This equivalence preserves temperedness
\cite[(2.11)]{Slo}, so the subcategory of $\Rep_f (\mh H)$ determined by $U \subset \mf t$
decomposes in the required way. In particular the Grothendieck group of that subcategory has
the desired property.

Let $\eta \in \R_{>0}$ and $\lambda_1,\lambda_2 \in \mf t_\R$ so that 
$\lambda_1 + \sqrt{-1} \lambda_2 \in U$. $\Rep_{f,t}^{[Q,\delta]}(\mh H)$ only has nonzero 
objects with $\mc O (\mf t)$-weights of this form if the central character of $\delta$ is 
$(W\Gamma)_Q \lambda_1$.

By \cite[Proposition 10.1]{SolGHA} and the complete decomposability of 
$\pi (Q,\delta,\lambda)$ for $\lambda \in \sqrt{-1}\mf t^Q_\R$ \cite[Proposition 7.2]{SolGHA}, there 
is a natural bijection between the set of irreducible objects of $\Rep_{f,t}^{[Q,\delta]}(\mh H)$ 
with central character $\lambda_1 + \sqrt{-1}\lambda_2$ and the analogous set for 
$\lambda_1 + \sqrt{-1}\eta\lambda_2$. Hence the proven property of the part of $R_t (\mh H)$
over $U$ extends to the whole of $R_t (\mh H)$.
\end{proof}  

The decomposition from Theorem \ref{thm:4.7} is available for $\mh H (\mf t,W\Gamma,\epsilon k,
\natural)$ for any $\epsilon \in \R$. It is compatible with $\zeta_\epsilon$ for $\epsilon > 0$,
but not for $\epsilon < 0$ (because $\zeta_{-1}$ does not preserve temperedness). 
Theorem \ref{thm:4.7} is hardly helpful in the case $\epsilon = 0$, because the parabolic subalgebras 
$\mh H (\mf t,(W\Gamma )_Q ,0,\natural)$ with $Q \neq \emptyset$ do not have any discrete series
representations. (The map $\zeta_0$ does not preserve the discrete series property.)

Let us consider a finite set of algebraic families of $\mh H$-representations
$\mf F (Q_i,\sigma_i)$ as in Definition \ref{def:2}. We assume that $\sigma_i$ is irreducible elliptic 
and that the representations
\begin{equation}\label{eq:4.7}
\pi (Q_i,\sigma_i, \lambda_i) = \mr{ind}_{\mh H^Q}^{\mh H} (\C_{\lambda_i} \otimes \sigma_i) \quad 
\lambda_i \in \mf t^{Q_i}, i = 1,\ldots, n_{\mf F}
\end{equation}
span $\Q \otimes_\Z R(\mh H)$. It follows from the Langlands classification for graded Hecke
algebras \cite{Eve} (which can be generalized to $\mh H$ with the method from \cite[\S 2.2]{SolAHA})
that every irreducible elliptic $\mh H$-representation is tempered and admits a 
$\mc O (\mf t)^{W\Gamma}$-character in 
\[
(\R \otimes_\Z \Z \Phi) \oplus \C \otimes_\R (\R \Phi^\vee)^\perp.
\]
In particular every $\sigma_i \in \Irr (\mh H_{Q_i})$ is tempered and admits a 
$\mc O (\mf t_{Q_i})^{(W \Gamma)_{Q_i}}$-character in $\R \Delta_Q$. Then $\pi (Q_i,\sigma,\lambda_i)$
is tempered if and only if $\lambda_i \in \sqrt{-1} \mf t^{Q_i}_\R$, see \cite[Lemma 2.2]{SolGHA}.
Every tempered family $\mf F^t (Q_i,\sigma_i)$ belongs to $R_t (\mh H)^{\mf d}$ for a unique
$\mf d = [Q,\delta] \in \Delta_{\mh H}$. We denote that as $i \prec \mf d$. In this situation we may
and will assume that $Q_i \supset Q$ and that $\sigma_i$ is a subquotient of 
$\mr{ind}_{\mh H^Q}^{\mh H^{Q_i}} \delta$.

By Theorem \ref{thm:4.1}, $\zeta_\epsilon (\sigma_i) = \sigma_i \circ m_\epsilon$ is a finite
dimensional elliptic representation of $\mh H (\mf t_{Q_i},(W \Gamma)_{Q_i}, \epsilon k, \natural)$.
Further, 
\[
\zeta_\epsilon (\pi (Q_i,\sigma_i,\lambda_i)) = \pi (Q_i, \zeta_\epsilon (\sigma_i),\lambda_i)
\qquad \text{for } \lambda_i \in \mf t^{Q_i} ,
\]
so $\zeta_\epsilon$ sends $\mf F (Q_i,\sigma_i)$ to the algebraic family $\mf F (Q_i,\zeta_\epsilon 
(\sigma_i))$. The bijectivity of $\zeta_\epsilon$ implies that the members of the algebraic 
families $\{ \mf F (Q_i, \sigma_\epsilon (\sigma_i)) \}_{i=1}^{n_{\mf F}}$ span\\ 
$\Q \otimes_\Z R (\mh H (\mf t,W \Gamma,\epsilon k,\natural))$. Thus Lemma \ref{lem:2.5} 
and Theorem \ref{thm:2.6} apply to the families $\mf F (Q_i, \zeta_0 (\sigma_i))$.

Recall from \eqref{eq:4.6} that we can identify the 
$HH_n (\mh H (\mf t, W\Gamma,\epsilon k,\natural))$ for all $\epsilon \in \C$ with one fixed 
vector space. For each $\epsilon \in \C$ we have the algebra homomorphisms
\begin{equation}\label{eq:4.28}
\mc F_{Q_i,\zeta_\epsilon (\sigma_i)} : \mh H (\mf t, W\Gamma,\epsilon k,\natural) \to
\mc O (\mf t^{Q_i}) \otimes \mr{End}_\C \big( \C[W \Gamma,\natural] \underset{\C[(W\Gamma)_Q,
\natural]}{\otimes} V_{\sigma_i} \big) .
\end{equation}

\begin{lem}\label{lem:4.5}
The map $\mc F_{Q_i,\zeta_\epsilon (\sigma_i)}$ is a homomorphism of 
$\mc O (\mf t)^{W \Gamma}$-algebras, for a module structure that depends on $\epsilon$.
In particular the induced map on Hochschild homology is $\mc O (\mf t)^{W \Gamma}$-linear.
\end{lem}
\begin{proof}
As $\mc O (\mf t)^{W\Gamma} \subset Z (\mh H (\mf t,W\Gamma,\epsilon k,\natural))$, it acts
naturally on $HH_n (\mh H (\mf t,W\Gamma,\epsilon k,\natural))$. By the irreducibility of 
$\sigma_i$, every $\zeta_\epsilon (\sigma_i)$ admits a 
$\mc O (t_{Q_i})^{(W\Gamma)_{Q_i})}$-character. It depends linearly on $\epsilon$.
By \cite[Theorem 6.4]{BaMo}, every $\pi (Q_i,\zeta_\epsilon (\sigma_i),\lambda_i)$ admits a 
$\mc O (\mf t)^{W\Gamma}$-character, which implies that the image of $\mc O (\mf t)^{W\Gamma}$
under $\mc F_{Q_i,\zeta_\epsilon (\sigma_i)}$ is central. That turns
\[
\mc O (\mf t^{Q_i}) \otimes \mr{End}_\C \big( \C[W \Gamma,\natural] \underset{\C[(W\Gamma)_Q,
\natural]}{\otimes} V_{\sigma_i} \big)
\] 
into a $\mc O (t)^{W\Gamma}$-algebra and makes
$\mc F_{Q_i,\zeta_\epsilon (\sigma_i)}$ $\mc O (\mf t)^{W\Gamma}$-linear. The final claim is one
of the functorial properties of Hochschild homology \cite[\S 1.1]{Lod}.
\end{proof}

The algebra homomorphisms \eqref{eq:4.28} and \eqref{eq:4.6} 
give rise to a family of maps 
\[
HH_n (\mc F_\epsilon) \;:\; \bigoplus\nolimits_{w \in \langle W \Gamma \rangle}
\big( \Omega^n (\mf t^w) \otimes \natural^w \big)^{Z_{W \Gamma}(w)} \to
\bigoplus\nolimits_{i=1}^{n_{\mf F}} \Omega^n (\mf t^{Q_i}) \qquad \epsilon \in \C.
\]
The discussion after \eqref{eq:4.6} and the construction of $\zeta_\epsilon$ entail
that these maps depend algebraically on $\epsilon$.

\begin{lem}\label{lem:4.2}
The map $HH_n (\mc F_\epsilon)$ is injective, for every $\epsilon \in \C$.
\end{lem}
\begin{proof}
Consider a nonzero element $x$ of the domain of $HH_n (\mc F_\epsilon)$. From Theorem \ref{thm:2.6}
we know that $HH_n (\mc F_0)$ is injective, so $HH_n (\mc F_0) x \neq 0$. As $HH_n (\mc F_\epsilon)x$ 
depends algebraically on $\epsilon \in \C$, it is nonzero when $|\epsilon|$ is sufficiently small.

For any $\eta \in \C^\times$ we have the algebra isomorphism
\[
m_\eta : \mh H (t,W\Gamma, \epsilon k,\natural) \to \mh H (t,W\Gamma, \eta^{-1} \epsilon k,\natural),
\]
which sends a family $\mf F (Q_i,\zeta_\epsilon (\sigma_i))$ to $\mf F (Q_i, \zeta_{\eta^{-1} 
\epsilon} (\sigma_i))$, with an additional scaling by $\eta$ on $\mf t^{Q_i}$. The latter respect
the entire structure, so we can conclude that $HH_n (\mc F_{\eta^{-1} \epsilon}) x \neq 0$ for all
$\eta \in \C^\times$.
\end{proof}

We want to analyse the maps $HH_n (\mc F_\epsilon)$ like in Proposition \ref{prop:2.8}. For
$g \in \langle W \Gamma \rangle$ we choose
$\lambda_{g,i} \in \C$ and $\phi_{g,i} : \mf t^g \to \mf t^{Q_i}$ as in Lemma \ref{lem:2.7}, so that
\begin{equation}\label{eq:4.11}
\nu_{g,v} = \sum\nolimits_{i=1}^{n_{\mf F}} \lambda_{g,i} \, \mr{tr} \, \pi (Q_i, \zeta_0 (\sigma_i),
\phi_{g,i}(v) ) \quad \text{in } \C \otimes_\Z R (\mh H (\mf t, W\Gamma, 0, \natural)) .
\end{equation}
Although $\nu_{g,v}$ has only been defined when $(W\Gamma)_v \cap Z_G (g) \subset \ker (\natural^g)$,
the right hand side of \eqref{eq:4.11} makes sense for any $v \in \mf t^g$. For $\epsilon \in \C$,
and $\mf d \in \Delta_{\mh H}$ we define the algebraic families of virtual 
$\mh H (\mf t, W\Gamma, \epsilon k, \natural)$-representations
\begin{equation}\label{eq:4.8}
\begin{aligned}
& \nu_{g,v}^{\epsilon, \mf d} = \sum\nolimits_{i=1, i \prec \mf d}^{n_{\mf F}} \lambda_{g,i} 
\, \mr{tr} \, \pi (Q_i, \zeta_\epsilon (\sigma_i), \phi_{g,i}(v) ) \qquad v \in \mf t^g,\\
& \nu_{g,v}^{\epsilon} = \sum\nolimits_{\mf d \in \Delta_{\mh H}} \nu_{g,v}^{\epsilon, \mf d} .
\end{aligned}
\end{equation}
For all $v \in \sqrt{-1} \mf t^g_\R, v_i \in \sqrt{-1} \mf t_\R^{Q_i}$ we have
\[
\nu^{0,\mf d}_{g,v}, \pi (Q_i, \zeta_0 (\sigma_i), v_i) \in \zeta_0 (R_t (\mh H)^{\mf d}) .
\]
Hence the proof of Lemma \ref{lem:2.7} can played entirely in 
$\Q \otimes_\Z \zeta_0 (R_t (\mh H)^{\mf d})$, which means that we do not need all elements of 
$W \Gamma$ for the $\phi_{g,i}$, only those of $(W\Gamma)_{\mf d}$. 

Like in \eqref{eq:2.35}, for $g \in \langle W \Gamma \rangle, h \in W \Gamma$ we define
\[
\lambda_{hgh^{-1},i} = \natural^g (h) \lambda_{g,i} \quad \text{and} \quad
\phi_{hgh^{-1},i} = \phi_{g,i} \circ h^{-1} .
\]
This is consistent with the equality $\nu_{hgh^{-1},hv} = \natural^g (h) \nu_{g,v}$ from
\eqref{eq:2.16}. Then we define $\nu_{hgh^{-1},v'}$ as in \eqref{eq:4.11} and 
$\nu_{hgh^{-1},v'}^\epsilon$ as in \eqref{eq:4.8}. The bijectivity in Theorem \ref{thm:4.1} 
implies that the $\nu_{hgh^{-1},v'}^\epsilon$ satisfy the same consistency relations. 

\begin{lem}\label{lem:4.6}
Let $g,h \in W \Gamma$ and $v \in \mf t^g$.
\enuma{
\item There is an equality 
\[
\nu^\epsilon_{hgh^{-1},hv} = \natural^g (h) \nu_{g,v}^\epsilon \quad \text{in} \quad 
\C \otimes_\Z R(\mh H (\mf t, W\Gamma, \epsilon k, \natural)).
\]
\item For $\mf d \in \Delta_{\mh H}$ we have
\[
\nu^{\epsilon,\mf d}_{hgh^{-1},h v} = \natural^g (h) \nu^{\epsilon,\mf d}_{g,v}.
\] 
}
\end{lem}
\begin{proof}
(a) By \eqref{eq:2.16} and the above conventions this holds when $g \in \langle W \Gamma \rangle$ 
and $(W \Gamma)_v \cap Z_{W \Gamma} (g) \subset \ker (\natural^g)$. That extends to all
$v \in V^g$ by continuity. 

For an arbitrary element $g' = w g w^{-1}$ of $G$, we get
\[
\nu^\epsilon_{hg'h^{-1},h v'} = \nu^\epsilon_{h w g w^{-1} h^{-1}, h v'} =
\natural^g (hw) \nu^\epsilon_{g, w^{-1} v'} =
\natural^g (hw) \natural^g (w)^{-1} \nu^\epsilon_{w g w^{-1},v'} .
\]
By Lemma \ref{lem:2.10}.b the right hand side equals $\natural^{g'}(h) \nu^\epsilon_{g',v'}$.\\
(b) Since every $\phi_{g,i}$ stabilizes $\sqrt{-1} \mf t_\R$, it makes sense to restrict our 
attention to the $\nu^{\epsilon,\mf d}_{g,v}$ with $v \in \sqrt{-1} \mf t_\R^g$.
These are precisely the $v$ for which the associated representations 
$\pi (Q_i,\sigma_i,\phi_{g,i}(v))$ are tempered.

Theorem \ref{thm:4.7} says that for various $\mf d$ the components $\nu^{\epsilon,\mf d}_{g,v}$
of $\nu^1_{g,v}$ live in linearly independent parts of $\Q \otimes_\Z R (\mh H)$.
Then the property of $\nu^\epsilon_{g,v}$ from Lemma \ref{lem:4.6}.a must also hold for 
all of its $\mf d$-components, so
\[
\nu^{\epsilon,\mf d}_{hgh^{-1},h v} = \natural^g (h) \nu^{\epsilon,\mf d}_{g,v}
\qquad v \in \mf t_\R^g .
\]
Both sides of this equality extend algebraically to $v \in \mf t^g$, so
the equality as well. 
\end{proof}

By Lemma \ref{lem:2.5} and Theorem \ref{thm:4.1}, the $\nu^\epsilon_{g,v}$ with $g \in \langle
W \Gamma \rangle$ and $(W \Gamma)_v \cap Z_{W \Gamma} (g) \subset \ker (\natural^g)$ span
$\C \otimes_\Z R (\mh H (\mf t, W \Gamma, \epsilon k, \natural))$, and a basis is obtained by
dividing out the $W \Gamma$-equivariance relations from Lemma \ref{lem:4.6}.a.

In general the virtual representations $\nu_{g,v}^\epsilon$ do not admit a central character,
but their summands $\nu_{g,v}^{\epsilon,\mf d}$ do: 

\begin{lem}\label{lem:4.9}
Let $cc(\delta) \in \mf t_{\R,Q}$ be an $\mc O (\mf t_Q)$-weight of $\delta$, in other words,
a representative of the central character of $\delta$. Let $g \in W \Gamma , v \in \mf t^g$.
Then the virtual $\mh H (\mf t, W \Gamma, \epsilon k, \natural)$-representation 
$\nu_{g,v}^{\epsilon,\mf d}$ admits the central character
\[
W \Gamma (\epsilon \, cc (\delta) + v) = W \Gamma ((W\Gamma)_Q e \, cc (\delta) + v) .
\]
\end{lem} 
\begin{proof}
For $\epsilon = 0$, by construction $\nu_{g,v}^0 = \nu_{g,v}$ has central character $W \Gamma v$.
Hence all the $\pi (Q_i, \zeta_0 (\sigma_i),\phi_{g,i}(v))$ with $\lambda_{g,i} \neq 0$ have
central character $W \Gamma v$. 

Consider $i \prec \mf d$. Since $\sigma_i$ is a subquotient of $\mr{ind}_{\mh H^Q}^{\mh H^{Q_i}} 
(\delta)$, $(W\Gamma)_{Q_i} cc(\delta)$ is the central character of $\sigma_i$. By the 
invertibility of intertwining operators for tempered parabolically induced representations and
the theory of R-groups for graded Hecke algebras \cite[\S 3.5 and 4.1]{SolHecke}, 
$\mr{ind}_{\mh H^Q}^{\mh H^{Q_i}} (\delta)$ is a direct sum of irreducible representations with 
exactly the same $\mc O (\mf t)$-weights. Thus every $\mc O (\mf t)$-weight of $\delta$, and in
particular $cc(\delta)$, is also a $\mc O (\mf t)$-weight of $\sigma_i$.

By Theorem \ref{thm:4.1}.iv, $\epsilon cc (\delta) + \phi_{g,i}(v)$ is an $\mc O (\mf t)$-weight
of $\zeta_\epsilon (\sigma_i) \otimes \phi_{g,i}(v)$. Then the central character of 
$\pi (Q_i,\zeta_\epsilon (\sigma_i),\phi_{g,i}(v))$ is
\begin{equation}\label{eq:4.24}
W \Gamma (\epsilon cc (\delta) + \phi_{g,i}(v)) = 
W \Gamma (\epsilon \phi_{g,i}^{-1} cc (\delta) + v) .
\end{equation}
Here $\phi_{g,i} \in (W\Gamma)_{\mf d}$, so $\phi_{g,i}^{-1} cc (\delta)$ is still a 
$\mc O (\mf t)$-weight of $\delta$. Thus all\\ $\pi (Q_i,\zeta_\epsilon (\sigma_i),\phi_{g,i}(v))$
with $\lambda_{g,i} \neq 0$ and $i \prec \mf d$ have the same central character \eqref{eq:4.24},
and so does their linear combination $\nu_{g,v}^{\epsilon,\mf d}$.
\end{proof}

\subsection{Hochschild homology} \ 

As $\phi_{g,i}$ is given by an element of $W\Gamma$, it extends naturally to a linear bijection
$\phi_{g,i} : \mf t \to \mf t$. Thus $\phi_{g,i} : \mf t^g \to \mf t^{Q_i}$ admits a one-sided inverse
\begin{equation}\label{eq:4.13}
\mf t^{Q_i} \xrightarrow{(\mr{id},0)} \mf t^{Q_i} \oplus \mf t_{Q_i} = 
\mf t \xrightarrow{\phi_{g,i}^{-1}} \mf t \to \mf t / (g-1) \mf t \cong \mf t^g .
\end{equation}
The algebra homomorphism $\phi_{g,i}^* : \mc O (\mf t^{Q_i}) \to \mc O (\mf t^g)$ also has a one-sided
inverse, namely composing functions with \eqref{eq:4.13}.
Like in \eqref{eq:2.18}, for each pair $(g,i)$ the map $\phi_{g_i}^*$ induces an algebra homomorphism
\begin{equation}\label{eq:4.16}
\phi_{g,i}^* : \mc O (\mf t^{Q_i}) \otimes \End_\C \big( \C [W \Gamma,\natural] 
\underset{\C [(W\Gamma)_Q,\natural}{\otimes} V_{\sigma_i} \big) \to
\mc O (V^g) \otimes \End_\C \big( \C [W \Gamma,\natural] 
\underset{\C [(W\Gamma)_Q,\natural}{\otimes} V_{\sigma_i} \big) .
\end{equation}
The map $\phi_{g,i}^*$ is $\mc O (\mf t^g)$-linear if we let $\mc O (\mf t^g)$ act on its domain via 
composition with \eqref{eq:4.13}. On the other hand, $\phi_{g,i}^*$ is $\mc O (\mf t)^{W\Gamma}$-linear 
if we endow both sides with the $\mc O (\mf t)^{W\Gamma}$-module structure coming from the central 
characters of the involved $\mh H (\mf t, W\Gamma, \epsilon k,\natural)$-representations. That works
for any $\epsilon \in \C$, but the resulting module structures depend on $\epsilon$.

The maps on Hochschild homology induced by the $\phi_{g,i}^*$ can be summed with coefficients 
$\lambda_{g,i}$, and that yields a map 
\begin{equation}\label{eq:4.10}
HH_n (\phi^*_{\mf d}) = \bigoplus_{g \in \langle W \Gamma \rangle} \sum_{i=1, i \prec \mf d}^{n_{\mf F}}
\lambda_{g,i} HH_n (\phi_{g,i}^*) \;: \bigoplus_{i=1, i \prec \mf d}^{n_{\mf F}} \Omega^n (\mf t^{Q_i}) 
\to \bigoplus_{g \in \langle W \Gamma \rangle} \Omega^n (\mf t^g) .
\end{equation}
By Lemma \ref{lem:4.9} it is $\mc O (\mf t)^{W\Gamma}$-linear if we let $\mc O (\mf t)^{W \Gamma}$ 
act on $\Omega^n (\mf t^g)$ via the central character of the underlying virtual representation 
$\nu_{g,v}^{\epsilon,\mf d}$. On the other hand, the map
\[
HH_n (\phi^*) = \bigoplus_{\mf d \in \Delta_{\mh H}} HH_n (\phi^*_{\mf d}) \;:\; 
\bigoplus_{i=1}^{n_{\mf F}} \Omega^n (\mf t^{Q_i}) \to \bigoplus_{g \in \langle W \Gamma \rangle} 
\Omega^n (\mf t^g) 
\]
is in general not $\mc O (\mf t)^{W\Gamma}$-linear for these module structures.

\begin{lem}\label{lem:4.3}
The map $HH_n (\phi^*)$ is injective. For each $\epsilon \in \C$, the image of 
$HH_n (\phi^*_{\mf d}) \circ HH_n (\mc F_\epsilon)$ is contained in $\bigoplus_{g \in \langle 
W \Gamma \rangle} (\Omega^n (\mf t^g) \otimes \natural^g )^{Z_{W \Gamma} (g)}$. 
\end{lem}
\begin{proof}
The injectivity can be shown in the same way as in Lemma \ref{lem:2.9}. Note that 
\[
HH_n (\phi^*_{\mf d}) \circ HH_n (\mc F_\epsilon) = \bigoplus\nolimits_{g \in \langle W \Gamma 
\rangle} \sum\nolimits_{i=1, i \prec \mf d}^{n_{\mf F}}
\lambda_{g,i} HH_n (\phi_{g,i}^* \circ \mc F_{Q_i, \zeta_\epsilon (\sigma_i)}) .
\]
The specialization of this expression at $(g,v)$ comes from the virtual representation 
$\nu^{\epsilon,\mf d}_{g,v}$ of $\mh H (\mf t,W\Gamma,\epsilon k,\natural)$. 
By Lemma \ref{lem:2.12} the map
\[
\mr{ev}_v \circ \mr{gtr} \circ \sum\nolimits_{i=1}^{n_{\mf F}} \lambda_{g,i} 
C_* \big( \phi_{g,i}^* \circ \mc F_{Q_i, \zeta_\epsilon (\sigma_i)} \big)
\]
cannot distinguish equivalent virtual representations. In combination with Lemma \ref{lem:4.6}.b we
find that the image of $\sum_{i=1}^{n_{\mf F}} \lambda_{g,i} HH_n ( \phi_{g,i}^* \circ \mc F_{Q_i, 
\zeta_\epsilon (\sigma_i)})$ consists of differential forms that transform as $(\natural^g )^{-1}$
under the action of $Z_{W \Gamma} (g)$.
\end{proof}

With the procedure from page \pageref{eq:2.38} we can achieve that our set of 
algebraic families of $\mh H$-representations $\mf F (Q_i,\sigma_i)$ minimally spans 
$\Q \otimes_\Z R(\mh H)$. Thus our new, reduced collection of algebraic families satisfies the
three properties listed on page \pageref{eq:2.38}.

We are ready to prove the description of the Hochschild homology of $\mh H$ in the style
of the trace Paley--Wiener theorem for reductive $p$-adic groups \cite{BDK}. Recall
that still all parameters $k_\alpha$ are real.

\begin{thm}\label{thm:4.4}
In the above setting we fix $\epsilon \in \C$. 
\enuma{
\item The map 
\[
HH_n (\mc F_\epsilon) : HH_n (\mh H (\mf t,W\Gamma,\epsilon k,\natural)) \to
\bigoplus\nolimits_{i=1}^{n_{\mf F}} \Omega^n (\mf t^{Q_i})
\]
is a $\mc O (\mf t)^{W \Gamma}$-linear injection, for the module structure from Lemma \ref{lem:4.5}.
\item $HH_n (\phi^*)$ is a bijection
\[
HH_n (\mc F_\epsilon) HH_n (\mh H (\mf t,W\Gamma,\epsilon k,\natural)) \to 
\bigoplus\nolimits_{g \in \langle W \Gamma \rangle} (\Omega^n (\mf t^g) 
\otimes \natural^g )^{Z_{W \Gamma} (g)} ,
\]
and it is $\mc O (\mf t)^{W\Gamma}$-linear with respect to the natural module structures.
\item In degree $n=0$ the condition on an element $\omega$ of $\bigoplus_{i=1}^{n_{\mf F}} \mc O
(\mf t^{Q_i})$ to belong to the image of $HH_0 (\mc F_\epsilon)$ is:
\begin{multline*}
\text{if } \mu_j \in \C, 1 \leq i_j \leq n_{\mf F}, \lambda_{i_j} \in \mf t^{Q_{i_j}} \text{ and } 
\sum\nolimits_j \mu_j \pi (Q_{i_j}, \zeta_\epsilon (\sigma_{i_j}),\lambda_{i_j}) = 0 \\
\text{ in }  \C \otimes_\Z R (\mh H (\mf t, W\Gamma,\epsilon k, \natural)), \text{ then } 
\sum\nolimits_j \mu_j \omega (\lambda_{i_j}) = 0 .
\end{multline*}
Equivalently, $\omega$ must determine a linear function on $\C \otimes_\Z R (\mh H (\mf t,
W\Gamma,\epsilon k, \natural))$, via the natural pairing \eqref{eq:2.42}. This yields
an isomorphism of $\mc O (\mf t)^{W\Gamma}$-modules
\[
HH_0 (\mh H (\mf t,W\Gamma,\epsilon k,\natural)) \cong \big( \C \otimes_\Z R (\mh H (\mf t,
W\Gamma,\epsilon k, \natural)) \big)^*_{\mr{reg}} .
\]
}
\end{thm}
\begin{proof}
(a) is just a restatement of Lemmas \ref{lem:4.5} and \ref{lem:4.2}.\\
(b) By Lemma \ref{lem:2.9} and the third property (from page \pageref{eq:2.37}) of our set of algebraic
families $\mf F (Q_i,\sigma_i)$, $HH_n (\phi^*)$ is injective.

Consider a finite subset $S$ of $\bigsqcup_{i=1}^{n_{\mf F}} \mf t^{Q_i}$ and let 
$I_S \subset \bigoplus_{i=1}^{n_{\mf F}} \mc O (\mf t^{Q_i})$ be the ideal of functions 
that vanish on $S$. For $m \in \N$ let $J_n^m$ (resp. $\tilde{J}^n_{m,\epsilon}$) be the image of
\begin{equation}\label{eq:4.12}
HH_n (\phi^*)^{-1} \Big( \bigoplus\nolimits_{g \in \langle W \Gamma \rangle} 
(\Omega^n (\mf t^g) \otimes \natural^g )^{Z_{W \Gamma} (g)} \Big)
\end{equation}
(resp. $HH_n (\mc F_\epsilon)$) in
\begin{equation}\label{eq:4.9}
\bigoplus\nolimits_{i=1}^{n_{\mf F}} \Omega^n (\mf t^{Q_i}) / I_S^m \Omega^n (\mf t^{Q_i}) .
\end{equation}
By Theorem \ref{thm:2.6} $\tilde{J}^n_{m,0} = J_m^n$. As $HH_n (\mc F_\epsilon)$
depends algebraically on $\epsilon \in \C$ and $J^n_m$ has finite dimension,
$\tilde{J}^n_{m,\epsilon} = J^n_m$ when $|\epsilon|$ is sufficiently small. The argument
with $m_\eta$ in the proof of Lemma \ref{lem:4.2} then shows that $\tilde{J}^n_{m,\epsilon}
= J^n_m$ for all $\epsilon \in \C$. 

Fix a $\mc O (\mf t)^{W\Gamma}$-character $W\Gamma \lambda$. Choose $S$ so that it contains
all $\lambda_i \in \mf t^{Q_i}$ for which $W \Gamma \lambda$ is the central character of 
$\pi (Q_i, \zeta_\epsilon (\sigma_i), \lambda_i)$. With Lemma \ref{lem:4.5} it 
follows from the above that the map $HH_n (\mc F_\epsilon)$ induces a surjection between the 
formal completions at $W\Gamma \lambda$ of the $\mc O (\mf t)^{W\Gamma}$-modules 
$HH_n (\mh H (\mf t,W\Gamma,\epsilon k,\natural))$ and \eqref{eq:4.12}.
Thus the quotient of \eqref{eq:4.12} by the image of $HH_n (\mc F_\epsilon)$ is a 
$\mc O (\mf t)^{W\Gamma}$-module $M$ all whose formal completions are 0. It is finitely 
generated because $\Omega^n (\mf t^{Q_i})$ is finitely generated as 
$\mc O (\mf t)^{W\Gamma}$-module. A general argument, which we formulate as  Lemma \ref{lem:1.1}
below, says that $M = 0$. Hence the image of $HH_n (\mc F_\epsilon)$ is as claimed. \\
(c) The description of the image of $HH_0 (\mc F_\epsilon)$ was proven in the case 
$\epsilon = 0$ in Theorem \ref{thm:2.6}. In combination with Theorem \ref{thm:4.1}, the same 
argument applies when $\epsilon \in \C^\times$. 
\end{proof}

Let $\mc O (V)$ be the ring of regular functions on a complex affine variety $V$. For each 
$v \in V$, let $I_v \subset \mc O (V)$ be the maximal ideal of functions vanishing at $v$. 
For any $\mc O (V)$-module $M$, we can form the completion
\[
\hat M_v = \varprojlim\nolimits_n M / I_v^n M .
\]

\begin{lem}\label{lem:1.1}
Let $M$ be a finitely generated $\mc O (V)$-module, such that $\hat M_v = 0$ for all $v \in V$.
Then $M = 0$. 
\end{lem}
\begin{proof}
For any $m \in M$, the image of $m$ in $\hat M_v$ is zero, so $m \in I_v^n M$ for all $n \in \N$.
Hence $M = I_v^n M$ for all $v \in V$ and all $n \in \N$.

As $M$ is finitely generated, we can write $M = \mc O (V)^r / N$ for some $\mc O (V)$-submodule
$N$ of $\mc O (V)^r$. In combination with the above we find
\[
\mc O (V)^r / N = I_v^n (\mc O (V)^r / N) = (I_v^n \mc O (V)^r + N ) / N
\]
Therefore $\mc O (V)^r = I_v^n \mc O (V)^r + N$ for all $v \in V$ and all $n \in \N$.
This is only possible when $N = \mc O (V)^r$, so $M = 0$.
\end{proof}

Like in \eqref{eq:2.36}, we can vary on \eqref{eq:4.10} and define the 
$\mc O (\mf t)^{W \Gamma}$-linear map
\begin{equation}\label{eq:4.19}
HH_n (\tilde \phi^*) = \bigoplus_{g \in W \Gamma} \sum_{i=1}^{n_{\mf F}} \lambda_{g,i}
HH_n (\phi_{g,i}^*) \;:\; \bigoplus_{i=1}^{n_{\mf F}} \Omega^n (\mf t^{Q_i}) \to 
\bigoplus_{g \in W \Gamma} \Omega^n (\mf t^g) .
\end{equation}
The same arguments as for Corollary \ref{cor:2.11} show that:

\begin{cor}\label{cor:4.10}
There is a $\C$-linear bijection
\[
HH_n (\tilde \phi^*) \circ HH_n (\mc F_\epsilon) \;:\; 
HH_n (\mh H (\mf t,W\Gamma,\epsilon k,\natural)) \to 
\Big( \bigoplus\nolimits_{g \in W \Gamma} \Omega^n (\mf t^g) \otimes \natural^g \Big)^{W \Gamma} .
\]
\end{cor}

We note that in Corollary \ref{cor:4.10} the target does not depend on $\epsilon$. In fact, in
Theorem \ref{thm:4.4} the map $HH_n (\phi^*)$ does not depend on $\epsilon$ either, and the
same goes for the subspace 
\[
HH_n (\mc F_\epsilon) HH_n (\mh H (\mf t, W\Gamma, \epsilon k, \natural)) \subset 
\bigoplus\nolimits_{i=1}^{n_{\mf F}} \Omega^n (\mf t^{Q_i}) ,
\]
Hence we can define a $\C$-linear bijection
\[
HH_n (\zeta_0) := HH_n (\mc F_1)^{-1} HH_n (\mc F_0) \;:\; 
HH_n (\mh H (\mf t, W\Gamma, 0, \natural)) \to HH_n (\mh H (\mf t, W\Gamma, k, \natural)).
\]

\begin{prop}\label{prop:4.13}
$HH_n (\zeta_0)$ is the unique $\C$-linear bijection
\[
HH_n (\mh H (\mf t, W\Gamma, 0, \natural)) \to HH_n (\mh H (\mf t, W\Gamma, k, \natural))
\]
such that
\[
HH_n (\mc F_{Q,\sigma}) \circ HH_n (\zeta_0) = HH_n (\mc F_{Q,\zeta_0 (\sigma)})
\]
for all algebraic families of $\mh H$-representations $\mf F (Q,\sigma)$.
\end{prop}
\begin{proof}
By construction
\begin{equation}\label{eq:4.25}
HH_n (\mc F_{Q_i,\sigma_i}) \circ HH_n (\zeta_0) = HH_n (\mc F_{Q_i,\zeta_0 (\sigma_i)})
\qquad i = 1,\ldots, n_{\mf F} .
\end{equation}
As $\mc F_1$ is built from the $\mc F_{Q_i, \sigma_i}$ and $\mc F_0$ from the
$\mc F_{Q_i, \zeta_0 (\sigma_i)}$, the property \eqref{eq:4.25} already determines 
$HH_n (\zeta_0)$ uniquely. 

It remains to check the condition for an arbitrary algebraic family $\mf F (Q,\sigma)$.
Recall that in Lemma \ref{lem:2.5} we exhibited a basis of $\C \otimes_\Z R (\mc O (\mf t)
\rtimes \C [W\Gamma, \natural])$, consisting of some virtual representations $\nu_{g,v}$.
The virtual representations $\nu_{g,v}^1 = \zeta_0^{-1} (\nu_{g,v})$ form a basis of
$\C \otimes_\Z R (\mh H)$. In \eqref{eq:4.11} we expressed $\nu_{g,v}$ as linear combination of
members of the families $\mf F (Q_i, \zeta_0 (\sigma_i))$, and by definition $\nu_{g,v}^1$ is
almost the same linear combination, only with $\mf F (Q_i, \sigma_i)$ instead. Write
\[
\mr{tr} \, \pi (Q,\sigma,\lambda) = \sum\nolimits_{g,v} c(g,v,\lambda) \nu_{g,v}^1, 
\]
then Theorem \ref{thm:4.1} implies
\[
\mr{tr} \, \pi (Q,\zeta_0(\sigma),\lambda) = \sum\nolimits_{g,v} c(g,v,\lambda) \nu_{g,v} .
\]
Hence there exist $c' (i,\lambda,v_i) \in \C$ such that
\begin{equation}\label{eq:4.27}
\begin{array}{lll}
\pi (Q,\sigma,\lambda) & = & \sum_{i,v_i} c' (i,\lambda,v_i) \, \pi (Q_i,\sigma_i, v_i) ,\\
\pi (Q,\zeta_0 (\sigma),\lambda) & = &
\sum_{i,v_i} c' (i,\lambda,v_i) \, \pi (Q_i,\zeta_0 (\sigma_i), v_i) ,
\end{array}
\end{equation}
in $\C \otimes_\Z R (\mh H)$ and $\C \otimes_\Z R(\mc O (\mf t) \rtimes \C [W\Gamma, \natural])$,
respectively. With \eqref{eq:4.25} we find that
\[
HH_n (\pi (Q,\sigma,\lambda)) \circ HH_n (\zeta) = HH_n (\pi (Q,\zeta_0 (\sigma),\lambda)) .
\]
The same reasoning for all $\lambda \in \mf t^Q$ simultaneously yields the required property
of $HH_n (\zeta_0)$.
\end{proof}

It turns out that the description of $HH_n (\mh H)$ in Theorem \ref{thm:4.4} decomposes further,
such that the decomposition reveals the structure as module over the centre. For 
$\mf d \in \Delta_{\mh H}$ we define $\mc F_{\mf d} = \bigoplus_{i \prec \mf d} \mc F_{Q_i,\sigma_i}$.

\begin{lem}\label{lem:4.8}
\enuma{
\item $HH_n (\mc F_1) HH_n (\mh H) = \bigoplus_{\mf d = [Q,\delta] \in \Delta_{\mh H}} 
HH_n (\mc F_{\mf d}) HH_n (\mh H)$.
\item The subspace
\[
HH_n (\mh H)^{\mf d} := HH_n (\mc F_1 )^{-1} HH_n (\mc F_{\mf d}) HH_n (\mh H) 
\]
of $HH_n (\mh H)$ can be defined canonically, 
without choosing any algebraic families of representations.
}
\end{lem}
\begin{proof}
(a) By definition the left hand side is contained in the right hand side. From Theorem \ref{thm:4.4}
we know the conditions that describe the left hand side: upon applying $HH_n (\phi^*)$ one lands
in $\bigoplus_{g \in \langle W \Gamma \rangle} (\Omega^n (V^g) \otimes \natural^g 
)^{Z_{W \Gamma} (g)}$. Recall from Lemma \ref{lem:4.3} that those conditions arose from the virtual 
$\mh H$-representations $\nu^1_{g,\lambda}$.

With Lemma \ref{lem:4.6}.b we see, in the same way as in the proof of Lemma \ref{lem:4.3}, 
that $HH_n (\phi_{\mf d}^*)$ sends the image of $HH_n (\mc F_1)$, or equivalently the image of 
$HH_n (\mc F_{\mf d})$, to $\bigoplus_{g \in \langle W \Gamma \rangle} (\Omega^n (\mf t^g)
\otimes \natural^g )^{Z_W \Gamma (g)}$. Hence, for any $x \in HH_n (\mh H)$:
\begin{align*}
& HH_1 (\mc F_1) x = \bigoplus\nolimits_{\mf d \in \Delta_{\mh H}} 
HH_n \big( \mc F_{\mf d} \big) x, \\
& HH_n \big( \mc F_{\mf d} \big) x \in HH_n (\mc F_1) HH_n (\mh H) . 
\end{align*}
(b) This subspace is well-defined by the injectivity of $HH_n (\mc F_1)$ (Theorem \ref{thm:4.4}.a).
Consider an algebraic family $\mf F (Q,\sigma)$ whose tempered part
\[
\mf F^t (Q,\sigma) = \{ \pi (Q,\sigma,\lambda) : \lambda \in \sqrt{-1} \mf t_\R^Q \}
\]
lies in $R_t (\mh H)^{\mf d'}$, for some $\mf d' \in \Delta_{\mh H} \setminus \{\mf d\}$. If we
express $\pi (Q,\sigma,\lambda)$ with $\lambda \in \sqrt{-1}\mf t_\R^Q$ as in \eqref{eq:4.27}, all
the coefficients $c' (i,\lambda,v_i)$ with $i \not\prec \mf d'$ are zero. By construction
\[
HH_n (\mc F_{Q_j,\sigma_j}) HH_n (\mh H)^{\mf d} = 0 \quad \text{if } j \prec \mf d'.
\]
Hence $HH_n (\mc F_{Q,\sigma}) HH_n (\mh H)^{\mf d}$ consists of algebraic differential forms 
on $\mf t^Q$, which vanish on $\sqrt{-1} \mf t_\R^Q$. Since $\sqrt{-1} \mf t_\R^Q$ is Zariski-dense
in $\mf t^Q$, 
\[
HH_n (\mc F_{Q,\sigma}) HH_n (\mh H)^{\mf d} = 0 .
\]
On the other hand, by Theorem \ref{thm:4.4} $HH_n (\mc F_{\mf d}) = \bigoplus_{i \prec \mf d}
HH_n (\mc F_{Q_i,\sigma_i})$ is injective on $HH_n (\mh H )^{\mf d}$. Thus $HH_n (\mh H)^{\mf d}$
can be characterized as
\begin{multline*}
\{ x \in HH_n (\mh H) : HH_n (\mc F_{Q,\sigma}) x = 0 \text{ for all algebraic families } 
\mf F (Q,\sigma) \text{ with } \\ \mf F^t (Q,\sigma) \text{ in } R_t (\mh H)^{\mf d'} 
\text{ for some } \mf d' \neq \mf d \} . \qedhere
\end{multline*}
\end{proof}

From Lemma \ref{lem:4.8} and Theorem \ref{thm:4.4} we conclude:

\begin{cor}\label{cor:4.12}
There is a canonical decomposition of $\mc O (\mf t)^{W \Gamma}$-modules
\[
HH_n (\mh H) = \bigoplus\nolimits_{\mf d \in \Delta_{\mh H}} HH_n (\mh H)^{\mf d}.
\]
The injection
\[
HH_n (\phi_{\mf d}^*) \circ HH_n (\mc F_{\mf d}) : HH_n (\mh H)^{\mf d} \to \bigoplus\nolimits_{g 
\in \langle W \Gamma \rangle} (\Omega^n (\mf t^g) \otimes \natural^g )^{Z_W \Gamma (g)}
\]
is $\mc O (\mf t)^{W \Gamma}$-linear if we let $\mc O (\mf t)^{W \Gamma}$ act at
$(g,v)$ via evaluation at the central character $W \Gamma (cc (\delta) + v)$ of $\nu_{g,v}^{1,\mf d}$.
\end{cor}

\begin{ex}
Consider the graded Hecke algebra $\mh H$ with $\mf t = \C^2$, $\Phi$ of type $A_2$ with
basis $\Delta = \{ \alpha = (1,0), \beta = (-1/2, \sqrt{3}/2) \}$ and parameters
$k_\alpha = k_\beta = k \in \R_{>0}$. The group $\Gamma$ and the 2-cocyle $\natural$ are trivial.
We have $W \cong S_3$, $\mf t_\alpha = \C \times \{0\}$, $\mf t^\alpha = \{ 0 \} \times \C$,
$\mf t^\emptyset = \mf t$ and $\mf t^\Delta = \{0\}$.
For each $\epsilon \in \C$ we need three algebraic families of 
$\mh H (\mf t, W, \epsilon k)$-representations, namely
\begin{itemize}
\item $\mf F (\emptyset, \mr{triv}) = \{ \pi (\emptyset, \mr{triv},\lambda) = 
\mr{ind}_{\mc O (\mf t)}^{\mh H (\mf t, W, \epsilon k)} (\C_\lambda) : \lambda \in \mf t \}$,
\item $\mf F (\{\alpha\},\mr{St}_\alpha)$, where the Steinberg representation of $\mh H_\alpha$
is defined by\\ $\mr{St}_\alpha |_{\C [W_\alpha]} = \mr{sgn}_{W_\alpha}$ and
$\mr{St}_\alpha |_{\mc O (\mf t_\alpha)} = \C_{-k}$,
\item $\mf F (\Delta, \mr{St})$, where the Steinberg representation of $\mh H$ is defined by\\
$\mr{St} |_{\C[W]} = \mr{sgn}_W$ and $\mr{St} |_{\mc O (\mf t)} = \C_{(-k,-\sqrt{3}k)}$.
\end{itemize}
Let us identify the virtual representations $\nu_{g,\lambda}^\epsilon$.
All maps $\phi_{g,i}$ are the identity, and the scalars $\lambda_{g,i}$ can be determined from
direct calculations in the algebra $\mc O (\mf t) \rtimes W = \mh H (\mf t,W,0)$. The latter
reduces further to a calculation in $\C[W]$ because $\mc O (\mf t)$ acts as evaluation at 0
on all the relevant representations.
\begin{itemize}
\item $\nu_{\mr{id},\lambda}^\epsilon = \mr{tr} \, \pi (\emptyset, \mr{triv},\lambda )$,
\item $\nu_{s_\alpha ,\lambda}^\epsilon = -\mr{tr} \, \pi (\{\alpha\},\mr{St}_\alpha ,0) + 
\mr{tr} \, \pi (\emptyset,\mr{triv},0) / 2$, because 
\[
\mr{tr} \, \mr{ind}_{W_\alpha}^W (\mr{sgn}_{W_\alpha}) + 
\mr{tr} \, \mr{ind}_{\{\mr{id}\}}^W (\mr{triv}) / 2
\]
is the trace function on $W$ associated to the conjugacy class of $s_\alpha$,
\item $\nu^\epsilon_{s_\alpha s_\beta,0} = \mr{tr} \, \mr{St} - \mr{tr} \, 
\pi (\{\alpha\},\mr{St}_\alpha ,0) + \mr{tr} \, \pi (\emptyset, \mr{triv},0) / 3$, because
\[
\mr{tr} \, \mr{sgn}_W - \mr{tr} \, \mr{ind}_{W_\alpha}^W (\mr{sgn}_{W_\alpha}) + 
\mr{tr} \, \mr{ind}_{\{\mr{id}\}}^W (\mr{triv}) / 3
\]
is the trace function on $W$ associated to the conjugacy class of $s_\alpha s_\beta$.
\end{itemize}
When $\epsilon = 0$ or $g = \mr{id}$, $\nu^\epsilon_{g,\lambda}$ has central character
$W \lambda$. In all other cases $\nu^\epsilon_{g,\lambda}$ does not admit a central character.

In this example $Z_W (s_\alpha) = \{\mr{id},s_\alpha\}$ acts trivially on $\mf t^\alpha$, and
$Z_W (s_\alpha s_\beta)$ acts trivially on $\mf t^\Delta = \{0\}$. Hence the components of 
$HH_n (\phi^*)$ indexed by $s_\alpha$ and $s_\alpha s_\beta$ do not impose any further restriction
on the image of $HH_n (\mc F_\epsilon)$. The component of $HH_n (\phi^*)$ indexed by id must have
image invariant under $W$, and by the expression for $\nu_{\mr{id},\lambda}^\epsilon$ that only
puts a condition on the image of $HH_n (\mc F_{\emptyset,\mr{triv}})$.
Thus Theorem \ref{thm:4.4} provides a bijection 
\begin{equation}\label{eq:4.15}
HH_n (\mc F_\epsilon) : HH_n (\mh H (\mf t, W, \epsilon k)) \to 
\Omega^n (\mf t)^W \oplus \Omega^n (\mf t^\alpha) \oplus \Omega^n (\{0\}) .
\end{equation}
The $\mc O (\mf t)^W$-module structure on the right hand is standard on $\Omega^n (\mf t)^W$, via 
evaluations at $(-\epsilon k,0) + \mf t^\alpha$ on $\Omega^n (\mf t^\alpha)$ and as evaluation
at $(-\epsilon k, -\epsilon \sqrt{3}k)$ on $\Omega^n (\{0\})$.  

The decomposition of $R_t (\mh H)$ from Theorem \ref{thm:4.7} has three direct summands, indexed
precisely by above three families. Here $R_t (\mh H)^{[P,\delta]}$ is spanned by
\[
\{ \pi (P,\delta,\lambda) : \lambda \in \sqrt{-1}\mf t_\R^P \} .
\]
For $\epsilon \in \R_{>0}$, \eqref{eq:4.15} is also the decomposition of 
$HH_n (\mh H (\mf t, W, \epsilon k))$ from Lemma \ref{lem:4.8}.a.\\
\end{ex}

\subsection{A Morita equivalent algebra} \ 
\label{par:Mor}

With applications to $p$-adic groups in mind we also consider some algebras that are Morita
equivalent to twisted graded Hecke algebras. Suppose that $W\Gamma$ is a subgroup of some finite
group $G$ and that $\natural$ extends to a 2-cocycle of $G$ (still denoted $\natural$). Then
$\C [W\Gamma ,\natural]$ is a subalgebra of $\C [G,\natural]$ and $G$ acts on the space
\[
V := G \times_{W\Gamma} \mf t
\]
by left multiplication. We fix a set of representatives $[G / W\Gamma] \subset G$ for $G / W\Gamma$,
with $1 \in [G / W \Gamma]$. Consider a $g \in [G / W\Gamma]$. In $g \mf t$ we have the root system 
$g (\Phi)$ with Weyl group $g W g^{-1} \subset G$. We define twisted graded Hecke algebra
\[
\mh H_g = \mh H (g \mf t, g W \Gamma g^{-1}, k^g, \natural ),
\]
where $k^g (g \alpha) = k (\alpha)$. By construction there is an algebra isomorphism
\[
\begin{array}{cccc}
\mr{Ad}(T_g) : & \mh H & \to & \mh H_g \\
& f T_w & \mapsto & (f \circ g^{-1}) T_g T_w T_g^{-1}
\end{array} \qquad f \in \mc O (\mf t), w \in W\Gamma .
\]
Next we define $\mh H (V,G,k,\natural)$ as the vector space
$\mc O (V) \otimes_\C \C [G,\natural]$ with the multiplication rules:
\begin{itemize}
\item $\mc O (V)$ and $\C [G,\natural]$ are embedded as unital subalgebras,
\item for each $g \in [G / W\Gamma]$, $\mh H_g$ is embedded as a subalgebra with underlying vector
space $\mc O (g \mf t) \otimes \C [g W \Gamma g^{-1},\natural]$,
\item if $g,\tilde g \in [G / W\Gamma]$ and $g \neq \tilde g$, then 
$h \tilde h = 0$ for all $h \in \mh H_g, \tilde h \in \mh H_{\tilde g}$,
\item $T_{\tilde g} T_g^{-1} h T_g T_{\tilde g}^{-1} = \mr{Ad}(T_{\tilde g}) \mr{Ad}(T_g)^{-1} h$
for $g,\tilde g \in [G / W\Gamma], h \in \mh H_g$.
\end{itemize}
It is easily checked that this yields an associative algebra, which in the case $k=0$ reduces
to $\mc O (V) \rtimes \C [G,\natural]$. The algebras $\mh H (V,G,k,\natural)$ and $\mh H$ are
Morita equivalent via the bimodules $1_{\mf t} \mh H (V,G,k,\natural)$ and 
$\mh H (V,G,k,\natural) 1_{\mf t}$. In particular 
\begin{equation}\label{eq:4.17}
\text{the inclusion } \mh H \to \mh H(V,G,k,\natural) \text{ is a Morita equivalence}
\end{equation}
and induces an isomorphism on Hochschild homology. We want to express \\
$HH_n (\mh H (V,G,k,\natural))$ so that all the subalgebras $\mh H_g$ participate on equal terms. 

The family of $\mh H$-representations $\mf F (Q_i,\sigma_i )$ gives rise to a family of 
$\mh H(V,G,k,\natural)$-representations (still denoted $\mf F (Q_i,\sigma_i)$) by applying 
\eqref{eq:4.17}, or equivalently by inducing from $\mh H$ to $\mh H(V,G,k,\natural)$. The
natural isomorphism
\begin{equation}\label{eq:4.18}
\mr{ind}_{\mh H}^{\mh H(V,G,k,\natural)} \pi (Q_i,\sigma_i,\lambda_i) \cong
\mr{ind}_{\mh H_g}^{\mh H(V,G,k,\natural)} \pi (g(Q_i),\mr{Ad}(T_g) \cdot \sigma_i,g(\lambda_i)) 
\end{equation}
shows that $\mf F (Q_i,\sigma_i)$ can also be obtained from the family of $\mh H_g$-representations
$\mf F (g(Q_i), \mr{Ad}(T_g) \cdot \sigma_i)$.
The $\mc O (\mf t)^{W\Gamma}$-algebra homomorphism $\mc F_{Q_i,\sigma_i}$ extends naturally to
$\mh H (G,V,k,\natural)$ with the same formulas, only now for the representations \eqref{eq:4.18}.
Similarly $HH_n (\mc F_1)$ extends naturally to
\[
HH_n (\mc F_1 ) : HH_n (\mh H (V,G,k,\natural)) \to 
\bigoplus\nolimits_{i=1}^{n_{\mf F}} \Omega^n (\mf t^{Q_i}).
\]
From the canonical isomorphism
\begin{equation}\label{eq:4.20}
\Omega^n (g^{-1}) : \bigoplus\nolimits_{i=1}^{n_{\mf F}} \Omega^n (\mf t^{Q_i}) \to
\bigoplus\nolimits_{i=1}^{n_{\mf F}} \Omega^n ((g\mf t)^{g (Q_i)}) 
\end{equation}
and \eqref{eq:4.18} we see that $HH_n (\mc F_1)$ arises by performing the analogous constructions
to the family of $\mh H_g$-representations $\mf F (g(Q_i), \mr{Ad}(T_g) \cdot \sigma_i)$.

For $g \in [G / W\Gamma], w \in W \Gamma, 1 \leq i \leq n_{\mf F}, v \in g(\mf t^w)$ we define 
\begin{align*}
& \phi_{g,w,i} = \phi_{w,i} \circ g^{-1} : g (\mf t^w ) \to \mf t^{Q_i} ,\\
& \lambda_{g,w,i} = \natural^w (g) \lambda_{w,i} ,\\
& \nu_{g,w,v} = \sum\nolimits_{i=1}^{n_{\mf F}} \lambda_{g,w,i} 
\mr{ind}_{\mh H}^{\mh H (V,G,k,\natural)} \pi (Q_i,\sigma_i,\phi_{g,w,i}(v)) .
\end{align*}
We let $G$ act on $[G / W\Gamma] \times W\Gamma$ by
\[
\tilde g (g,w) = (h, \tilde w w \tilde{w}^{-1}) \quad \text{if} \quad
\tilde g g = h \tilde w \text{ with } h \in [G / \Gamma] , \tilde w \in W \Gamma .
\]

\begin{lem}\label{lem:4.14}
For $h \in G, g \in [G / W\Gamma], w \in W \Gamma$ and $v \in g (\mf t^w)$:
\[
\nu_{h(g,w),hv} = \natural^{g w g^{-1}}(h) \nu_{g,w,v} .
\]
\end{lem}
\begin{proof}
By definition $\nu_{g,w,v} = \natural^w (g) \nu_{1,w,g^{-1}v}$. As
\[
\nu_{1,w,v} = \mr{ind}_{\mh H}^{\mh H (V,G,k,\natural)} (\nu^1_{w,v})
\]
and $\mr{ind}_{\mh H}^{\mh H (V,G,k,\natural)}$ is a Morita equivalence, the $\nu_{1,w,v}$
satisfy the same relations as the $\nu_{w,v}^1$. In particular, by Lemma \ref{lem:4.6}.a
$\nu_{1,\gamma w \gamma^{-1},hv} = \natural^w (\gamma) \nu_{1,w,v}$ for
$\gamma \in W\Gamma$. From these properties we deduce, for $h = \tilde g \tilde w \in G$:
\begin{align*}
\nu_{h (1,w),hv} = \nu_{\tilde g, \tilde w w \tilde{w}^{-1},\tilde g \tilde w v} & = 
\natural^{\tilde w w \tilde{w}^{-1}} (\tilde g) \nu_{1,\tilde w w \tilde{w}^{-1},\tilde w v} \\
& = \natural^{\tilde w w \tilde{w}^{-1}} (\tilde g) \natural^w (\tilde w) 
\nu_{1,w,v} = \natural^w (h) \nu_{1,w,v} ,
\end{align*}
where the last step relies on Lemma \ref{lem:2.10}. It follows that
\[
\nu_{h (g,w),hv} = \nu_{hg(1,w),hv} = \natural^w (hg) \nu_{1,w,g^{-1}v} =
\natural^w (hg) \natural^w (g)^{-1} \nu_{g,w,v} =
\natural^{g w g^{-1}} (h) \nu_{g,w,v} ,
\]
where we used Lemma \ref{lem:2.10} again. 
\end{proof}

The version of \eqref{eq:2.36} and \eqref{eq:4.19} for $\mh H (V,G,k,\natural)$ is the map
$HH_n (\tilde \phi^*)$ defined as
\[
\bigoplus_{g \in [G / W\Gamma]} \bigoplus_{w \in W \Gamma} \sum_{i=1}^{n_{\mf F}}
\lambda_{g,w,i} HH_n (\phi_{g,w,i}^*) \;:\; \bigoplus_{i=1}^{n_{\mf F}} \Omega^n (\mf t^{Q_i}) \to 
\bigoplus_{g \in [G / W\Gamma]} \bigoplus_{w \in W \Gamma} \Omega^n (g \cdot \mf t^w) .
\]
By \eqref{eq:4.18} and \eqref{eq:4.20}, all the subalgebras $\mh H_g$ are involved in the same
way in $HH_n (\tilde \phi^*)$. The map $HH_n (\tilde \phi^*)$ is injective for the same reasons
as $HH_n (\phi^*)$, see Lemmas \ref{lem:2.9} and \ref{lem:4.3}.

\begin{prop}\label{prop:4.11}
\enuma{
\item The map
\[
HH_n (\mc F_1) : HH_n (\mh H(V,G,k,\natural)) \to 
\bigoplus\nolimits_{i=1}^{n_{\mf F}} \Omega^n (\mf t^{Q_i})
\]
is a $\mc O (V)^G$-linear injection. 
\item The $\C$-linear map 
\[
HH_n (\tilde \phi^*) : HH_n (\mc F_1) HH_n (\mh H(V,G,k,\natural)) \to
\Big( \bigoplus_{g \in [G / W\Gamma]} \bigoplus_{w \in W \Gamma} 
\Omega^n (g \cdot \mf t^w) \otimes \natural^{g w g^{-1}} \Big)^G 
\]
is bijective.
\item $HH_0 (\mc F_1) HH_0 (\mh H(V,G,k,\natural))$ equals the set of $\sum_{i=1}^{n_{\mf F}} \omega_i
\in \bigoplus\nolimits_{i=1}^{n_{\mf F}} \mc O (\mf t^{Q_i})$ for which the map
\[
\mr{ind}_{\mh H}^{\mh H (V,G,k,\natural)} \pi (Q_i,\sigma_i,\phi_{g,w,i}(v)) \;
\mapsto \; \omega_i (\lambda_i) \qquad i = 1,\ldots,n_{\mf F}, \lambda_i \in \mf t^{Q_i}
\]
descends to a linear function on $\C \otimes_\Z R (\mh H (V,G,k,\natural))$. This provides an
isomorphism of $\mc O (V)^G$-modules
\[
HH_0 (\mc F_1) HH_0 (\mh H(V,G,k,\natural)) \cong 
\big( \C \otimes_\Z R (\mh H(V,G,k,\natural)) \big)^*_{\mr{reg}} .
\]
}
\end{prop}
\begin{proof}
(a) Lemma \ref{lem:4.5} says that $HH_n (\mc F_1)$ is a homomorphism of modules over 
$\mc O (V)^G = \mc O (\mf t)^{W\Gamma}$. By Lemma \ref{lem:4.2} and \eqref{eq:4.17}, it is injective.\\
(b) From Corollary \ref{cor:4.10} we know that the projection of the image of \\
$HH_n (\tilde \phi^*) HH_n (\mc F_1)$ on the summands indexed by $g = 1$ is precisely
\begin{equation}\label{eq:4.22}
\big( \bigoplus\nolimits_{w \in W \Gamma} \Omega^n (\mf t^w) \otimes \natural^{w} \big)^{W \Gamma} .
\end{equation}
By Lemma \ref{lem:4.14} this image consists of $G$-invariant elements. Hence every element in the
image of $HH_n (\tilde \phi^*) HH_n (\mc F_1)$ is determined by its summands with $g = 1$. From
the natural isomorphism of \eqref{eq:4.22} with the asserted image (via removing the summands 
with $g \neq 1$) we see that $HH_n (\tilde \phi^*) HH_n (\mc F_1)$ indeed has that image.\\
(c) This follows from Theorem \ref{thm:4.4} and the Morita equivalence \eqref{eq:4.17}.
\end{proof}

We define a $\mh H (V,G,k,\natural)$-representation $\pi$ to be tempered if the 
$\mh H$-representation $1_{\mf t} \pi$ is tempered. The decomposition from Theorem \ref{thm:4.7}
also holds for the category of finite dimensional tempered $\mh H (V,G,k,\natural)$-representations,
by the Morita equivalence \eqref{eq:4.17}. Hence Lemmas \ref{lem:4.6}.b and \ref{lem:4.8} pertain
to $\mh H (V,G,k,\natural)$ as well. Consequently there is a canonical decomposition like in
Corollary \ref{cor:4.12}:
\begin{equation}\label{eq:4.23} 
\begin{array}{lll}
HH_n (\mh H (V,G,k,\natural))^{\mf d} & = & HH_n (\mc F_1)^{-1} HH_n (\mc F_{\mf d})
HH_n (\mh H (V,G,k,\natural))^{\mf d} , \\
HH_n (\mh H (V,G,k,\natural)) & = & 
\bigoplus\nolimits_{\mf d \in \Delta_{\mh H}} HH_n (\mh H (V,G,k,\natural))^{\mf d} .
\end{array}
\end{equation}
Furthermore $HH_n (\tilde \phi^*) \circ HH_n (\mc F_{\mf d})$ is $\mc O (V)^G$-linear on
$HH_n (\mh H (V,G,k,\natural))^{\mf d}$ if we endow the target with the module structure coming 
from the central characters of the virtual $\mh H (V,G,k,\natural)$-representations 
\begin{equation}\label{eq:4.26}
\nu_{g,w,v}^{\mf d} = \sum\nolimits_{i=1, i \prec \mf d}^{n_{\mf F}} \lambda_{g,w,i} \,
\mr{tr} \, \mr{ind}_{\mh H}^{\mh H (V,G,k,\natural)} \pi (Q_i,\sigma_i,\phi_{g,w,i}(v)) .
\end{equation}

\end{document}